\documentclass[11pt,reqno,twoside]{article}
\RequirePackage[OT1]{fontenc}
\RequirePackage{amsthm,amsmath}
\RequirePackage{hypernat}
\usepackage{color,epsfig}
\usepackage{amsfonts,amstext,amssymb,amscd}
\usepackage{mathrsfs}
\usepackage{upgreek}
\usepackage{tgtermes}

\usepackage[colorlinks=true]{hyperref}
\hypersetup{
    colorlinks,
    linkcolor=red,          
    citecolor=blue,        
    filecolor=magenta,      
    urlcolor=blue,           
}

\usepackage{fullpage}
\usepackage{enumitem}
\usepackage[section]{algorithm}
\usepackage{algorithmic}
\usepackage{graphicx,graphics}
\usepackage{multirow,multicol}
\usepackage{mysty}

\SetLabelAlign{parright}{\parbox[t]{\labelwidth}{\raggedleft{#1}}}
\setlist[description]{style=multiline,topsep=4pt,align=parright}

\numberwithin{equation}{section}

\newcommand*\samethanks[1][\value{footnote}]{\footnotemark[#1]\,\,}

\title{Sharp Oracle Inequalities for Low-complexity Priors}
\author{Tung Duy Luu\thanks{Normandie Univ, ENSICAEN, CNRS, GREYC, France, Email: \{duy-tung.luu, ~Jalal.Fadili\}@ensicaen.fr.} \and Jalal Fadili\samethanks \and Christophe Chesneau\thanks{Normandie Univ, UNICAEN, CNRS, LMNO, France, Email: christophe.chesneau@unicaen.fr.}}

\date{}

\begin{document}

\maketitle

\begin{abstract}
In this paper, we consider a high-dimensional statistical estimation problem in which the the number of parameters is comparable or larger than the sample size. We present a unified analysis of the performance guarantees of exponential weighted aggregation and penalized estimators with a general class of data losses and priors which encourage objects which conform to some notion of simplicity/complexity. More precisely, we show that these two estimators satisfy sharp oracle inequalities for prediction ensuring their good theoretical performances. We also highlight the differences between them. When the noise is random, we provide oracle inequalities in probability using concentration inequalities. These results are then applied to several instances including the Lasso, the group Lasso, their analysis-type counterparts, the $\ell_\infty$ and the nuclear norm penalties. All our estimators can be efficiently implemented using proximal splitting algorithms.
\end{abstract}


\begin{AMS}
{62G07}
{62G20}
\end{AMS}


\section{Introduction} \label{sec:intro}

\subsection{Problem statement} \label{subsec:ps}

Our statistical context is the following. Let $\yy=(\yy_1,\yy_2,\cdots,\yy_n)$ be $n$ identically distributed observations with common marginal distribution, and $\XX \in \RR^{n \times p}$ a deterministic design matrix. The goal to estimate a parameter vector $\xx \in \RR^p$ of the observations marginal distribution based on the data $\yy$ and $\XX$. 

Let $F: \RR^n \times \RR^n \to \RR$ be a loss function supposed to be smooth and convex that assigns to each $\xx \in \RR^p$ a cost $F(\XX\xx,\yy)$. Let $\xx_0 \in \Argmin_{\xx \in \RR^p}\EE{}{F(\XX\xx,\yy)}$ be any minimizer of the population risk. We regard $\xx_0$ as the true parameter. A usual instance of this statistical setting is the standard linear regression model based on $n$ pairs $(\yy_i,\XX_i)$ of response-covariate that are linked linearly $\yy=\XX\xx_0+\bsxi$, and $F(\bsu,\yy)=\tfrac{1}{2}\norm{\yy-\bsu}{2}^2$. 


Our goal is to provide general oracle inequalities in prediction for two estimators of $\xx_0$: the penalized estimator and exponential weighted aggregation. In the setting where "$p$ larger than $n$ (possibly much larger), the estimation problem is ill-posed since the rectangular matrix $\XX$ has a kernel of dimension at least $p-n$. To circumvent this difficulty, we will exploit the prior that $\xx_0$ has some low-complexity structure (among which sparsity and low-rank are the most popular). That is, even if the ambient dimension $p$ of $\xx_0$ is very large, its intrinsic dimension is much smaller than the sample size $n$. This makes it possible to build estimates $\XX\widehat{\xx}$ with good provable performance guarantees under appropriate conditions. There has been a flurry of research on the use of low-complexity regularization in ill-posed recovery problems in various areas including statistics and machine learning.

\subsection{Variational/Penalized Estimators}
Regularization is now a central theme in many fields including statistics, machine learning and inverse problems. It allows one to impose on the set of candidate solutions some prior structure on the object to be estimated. This regularization ranges from squared Euclidean or Hilbertian norms to non-Hilbertian norms (e.g. $\ell_1$ norm for sparse objects, or nuclear norm for low-rank matrices) that have sparked considerable interest in the recent years. In this paper, we consider the class of estimators obtained by solving the convex optimization problem\footnote{To avoid trivialities, the set of minimizers is assumed non-empty, which holds for instance if $J$ is also coercive.}
\begin{equation} \label{eq:minP}
\thvar \in \Argmin_{\xx \in \RR^p} \ens{V_n(\xx) \eqdef \tfrac{1}{n}F(\XX\xx,\yy) + \lambda_n J(\xx)} ,
\end{equation}
where the regularizing penalty $J$ is a proper closed convex function that promotes some specific notion of simplicity/low-complexity, and $\lambda_n > 0$ is the regularization parameter. A prominent member covered by \eqref{eq:minP} is the Lasso \cite{chen1999atomi,tibshirani1996regre,osborne2000new,donoho2006most,Candes09,BickelLassoDantzig07,BuhlmannVandeGeerBook11,koltchinskii08} and its variants such the analysis/fused Lasso \cite{rudin1992nonlinear,tibshirani2005sparsity}, SLOPE~\cite{CandesSLOPE14,CandesSLOPE15} or group Lasso \cite{bakin1999adaptive,yuan2006model,bach2008consistency,Wei10}. Another example is the nuclear norm minimization for low rank matrix recovery motivated by various applications including robust PCA, phase retrieval, control and computer vision~\cite{recht2010guaranteed,candes2009exact,fazel2001rank,CandesPhaseLift}. See \cite{negahban2010unified,BuhlmannVandeGeerBook11,vandegeer14,vaiterbookchap15} for generalizations and comprehensive reviews. 

\subsection{Exponential Weighted Aggregation (EWA)}
An alternative to the the variational estimator \eqref{eq:minP} is the aggregation by exponential weighting, which consists in substituting averaging for minimization. The aggregators are defined via the probability density function
\begin{equation} \label{eq:EWAaggregator}
\mu_n(\xx) = \frac{\exp\pa{-V_n(\xx)/\beta}}{\int_\Theta \exp\pa{-V_n(\bsom)/\beta}d\bsom},
\end{equation}
where $\beta>0$ is called temperature parameter. If all $\xx$ are candidates to estimate the true vector $\xx_0$, then $\Theta = \RR^p$. The aggregate is thus defined by
\begin{equation} \label{eq:EWAaggregate}
\thewa = \int_{\RR^p} \xx\mu_n(\xx)d\xx.
\end{equation}

Aggregation by exponential weighting has been widely considered in the statistical and machine learning literatures, see e.g. \cite{dala07,dala08,dala09,dala12-1,nemi00,yang04,rigo07,lecu07,alquier13,luu16} to name a few. $\thewa$ can also be interpreted as the posterior conditional mean in the Bayesian sense if $F/(n\beta)$ is the negative-loglikelihood associated to the noise $\bsxi$ with the prior density $\pi(\xx)\propto \exp\pa{-\lambda_n J(\xx)/\beta}$.

\subsection{Oracle inequalities}
Oracle inequalities, which are at the heart of our work, quantify the quality of an estimator compared to the best possible one among a family of estimators. These inequalities are well adapted in the scenario where the prior penalty promotes some notion of low-complexity (e.g. sparsity, low rank, etc.). 
Given two vectors $\xx_1$ and $\xx_2$, let $R_n\pa{\xx_1,\xx_2}$ be a nonnegative error measure between their predictions, respectively $\XX\xx_1$ and $\XX\xx_2$. A popular example is the averaged prediction squared error $\tfrac{1}{n}\norm{\XX\xx_1 -  \XX\xx_2}{2}^2$, where $\norm{\cdot}{2}$ is the $\ell_2$ norm. $R_n$ will serve as a measure of the performance of the estimators $\thewa$ and $\thvar$.
More precisely, we aim to prove that $\thewa$ and $\thvar$ mimic as much as possible the best possible model. This idea is materialized in the following type of inequalities (stated here for EWA)
\begin{equation} \label{eq:oi}
R_n\bpa{\thewa,\xx_0} \leq C\inf_{\xx\in\RR^p} \bpa{R_n\pa{\xx,\xx_0} + \Delta_{n,p,\lambda_n,\beta}(\xx)},
\end{equation}
where $C\geq 1$ is the leading constant of the oracle inequality and the remainder term $\Delta_{n,\lambda_n,\beta}(\xx)$ depends on the performance of the estimator, the complexity of $\xx$, the sample size $n$, the dimension $p$, and the regularization and temperature parameters $(\lambda_n,\beta)$. An estimator with good oracle properties would correspond to $C$ close to $1$ (ideally, $C=1$, in which case the inequality is said ``sharp''), and $\Delta_{n,p,\lambda_n,\beta}(\xx)$ is small and decreases rapidly to $0$ as $n \to +\infty$.

\subsection{Contributions}

We provide a unified analysis where we capture the essential ingredients behind the low-complexity priors promoted by $J$, relying on sophisticated arguments from convex analysis and our previous work \cite{Fadilidecomposable13,vaiterimaiai13,vaiterps14,vaiteraism16,vaiterbookchap15}. Our main contributions are summarized as follows:
\begin{itemize}
	\item We show that the EWA estimator $\thewa$ in \eqref{eq:EWAaggregator} and the variational/penalized estimator $\thvar$ in \eqref{eq:minP} satisfy (deterministic) sharp oracle inequalities for prediction with optimal remainder term, for general data losses $F$ beyond the usual quadratic one, and $J$ is a proper finite-valued sublinear function (i.e. $J$ is finite-valued convex and positively homogeneous). We also highlight the differences between the two estimators in terms of the corresponding bounds.
	\item When the observations are random, we prove oracle inequalities in probability. The theory is non-asymptotic in nature, as it yields explicit bounds that hold with high probability for finite sample sizes, and reveals the dependence on dimension and other structural parameters of the model.
	\item For the standard linear model with Gaussian or sub-Gaussian noise, and a quadratic loss, we deliver refined versions of these oracle inequalities in probability. We underscore the role of the Gaussian width, a concept that captures important geometric characteristics of sets in $\RR^n$.
	\item These results yield naturally a large number of corollaries when specialized to penalties routinely used in the literature, among which the Lasso, the group Lasso, their analysis-type counterparts (fused (group) Lasso), the $\ell_\infty$ and the nuclear norms. Soem of these corollaries are known and others novel.
\end{itemize}

The estimators $\thewa$ and $\thvar$ can be easily implemented thanks to the framework of proximal splitting methods, and more precisely forward-backward type splitting. While the latter is well-known to solve \eqref{eq:minP} \cite{vaiterbookchap15}, its application within a proximal Langevin Monte-Carlo algorithm to compute $\thewa$ with provable guarantees has been recently developed by the authors in \cite{luu16} to sample from log-semiconcave densities\footnote{In a forthcoming paper, this framework was extended to cover the even more general class of prox-regular functions.}, see also \cite{Durmus16} for log-concave densities.

\subsection{Relation to previous work}
Our oracle inequality for $\thewa$ extends the work of \cite{dalalyan16} with an unprecedented level of generality, far beyond the Lasso and the nuclear norm. Our prediction sharp oracle inequality for $\thvar$ specializes to that of \cite{sunzhang} in the case of the Lasso (see also the discussion in \cite{dalalyan17} and references therein) and that of \cite{koltchinskii11} for the case of the nuclear norm. Our work also goes much beyond that in \cite{vandegeer14} on weakly decomposable priors, where we show in particular that there is no need to impose decomposability on the regularizer, since it is rather an intrinsic property of it.

\subsection{Paper organization}
Section~\ref{sec:estimlc} states our main assumptions on the data loss and the prior penalty. All the concepts and notions are exemplified on some penalties some of which are popular in the literature. In Section~\ref{sec:soi}, we prove our main oracle inequalities, and their versions in probability. We then tackle the case of linear regression with quadratic data loss in Section~\ref{sec:SOIlin}. Concepts from convex analysis that are essential to this work are gathered in Section~\ref{sec:convana}. A key intermediate result in the proof of our main results is established in Section~\ref{sec:avdotp} with an elegant argument relying on Moreau-Yosida regularization.

\subsection{Notations}
\label{sec:notations}

\paragraph{Vectors and matrices} 
For a $d$-dimensional Euclidean space $\RR^d$, we endow it with its usual inner product $\dotp{\cdot}{\cdot}$ and associated norm $\anorm{\cdot}{2}$. $\Id_d$ is the identity matrix on $\RR^d$. For $p \geq 1$, $\anorm{\cdot}{p}$ will denote the $\ell_p$ norm of a vector with the usual adaptation for $p=+\infty$.

In the following, if $T$ is a vector space, $\proj_T$ denotes the orthogonal projector on $T$, and
\begin{equation*}
  \xx_T = \proj_T \xx \qandq \XXT = \XX \proj_T.
\end{equation*}
For a finite set $\calC$ we denote $\abs{\calC}$ its cardinality. For $I \subset \ens{1,\dots,p}$, we denote by $I^c$ its complement. $\xx_{I}$ is the subvector whose entries are those of $\xx$ restricted to the indices in $I$, and $\XX_{I}$ the submatrix whose columns are those of $\XX$ indexed by $I$. For any matrix $\XX$, $\XX^\top$ denotes its transpose and $\XX^+$ its Moore-Penrose pseudo-inverse. For a linear operator $\bs A$, $\bs A^*$ is its adjoint.

\paragraph{Sets}
For a nonempty set $\calC \in \RR^p$, we denote $\co{\calC}$ the closure of its convex hull, and $\iota_\calC$ its indicator function, i.e. $\iota_\calC(\xx)=0$ if $\xx \in \calC$ and $+\infty$ otherwise. For a nonempty convex set $\calC$, its \emph{affine hull} $\Aff(\calC)$ is the smallest affine manifold containing it. It is a translate of its \emph{parallel subspace} $\Lin(\calC)$, i.e. $\Lin(\calC) = \Aff(\calC) - \xx = \RR(\calC-\calC)$; for any $\xx \in \calC$. The \emph{relative interior} $\ri(\calC)$ of a convex set $\calC$ is the interior of $\calC$ for the topology relative to its affine full.

\paragraph{Functions} 
A function  $f: \RR^p \to \RR \cup \{+\infty\}$ is closed (or lower semicontinuous) if so is its epigraph. It is coercive if $\lim_{\anorm{\xx}{2} \to +\infty}f(\xx)=+\infty$, and strongly coercive if $\lim_{\anorm{\xx}{2} \to +\infty}f(\xx)/\anorm{x}{2}=+\infty$. The effective domain of $f$ is $\dom(f) = \enscond{\xx\in\RR^p}{f(\xx) < +\infty}$ and $f$ is proper if $\dom(f) \neq \emptyset$ as is the case when it is finite-valued. A function is said sublinear if it is convex and positively homogeneous. The Legendre-Fenchel conjugate of $f$ is $f^*(\bs z)=\sup_{\xx \in \RR^p} \dotp{\bs z}{\xx} - f(\xx)$. For $f$ proper, the functions $(f,f^*)$ obey the Fenchel-Young inequality
\begin{equation}
\label{eq:fenineq}
f(\xx) + f^*(\bs z) \geq \dotp{\bs z}{\xx}, \quad \forall (\xx,\bs z) \in \RR^p \times \RR^p .
\end{equation}
When $f$ is a proper lower semicontonuous and convex function, $(f,f^*)$ is actually the best pair for which this inequality cannot be tightened. 
For a function $g$ on $\RR_+$, the function $g^+: a \in \RR_+ \mapsto g^+(a)=\sup_{t \geq 0} a t - g(t)$ is called the monotone conjugate of $g$. The pair $(g,g^+)$ obviously obeys \eqref{eq:fenineq} on $\RR_+ \times \RR_+$.

For a $C^1$-smooth function $f$, $\nabla f(\xx)$ is its (Euclidean) gradient. For a bivariate function $g: (\bseta,\yy) \in \RR^n \times \RR^n \to \RR$ that is $C^2$ with respect to the first variable $\bseta$, for any $\yy$, we will denote $\nabla g(\bseta,\yy)$ the gradient of $g$ at $\bseta$ with respect to the first variable.

The \emph{subdifferential} $\partial f(\xx)$ of a convex function $f$ at $\xx$ is the set
\[
	\partial f(\xx) = \enscond{\bseta \in \RR^p}{f(\xx') \geq f(\xx) + \dotp{\bseta}{\xx'-\xx}, \quad \forall \xx' \in \dom(f)} ~.
\]
An element of $\partial f(\xx)$ is a subgradient. If the convex function $f$ is differentiable at $\xx$, then its only subgradient is its gradient, i.e. $\partial f(\xx) = \ens{\nabla f(\xx)}$.

The \emph{Bregman divergence} associated to a convex function $f$ at $\xx$ with respect to $\bseta \in \partial f(\xx) \neq \emptyset$ is
\[
\breg{f}{\bseta}{\bxx}{\xx} = f(\bxx) - f(\xx) - \dotp{\bseta}{\bxx-\xx} .
\]
The Bregman divergence is in general nonsymmetric. It is also nonnegative by convexity. When $f$ is differentiable at $\bxx$, we simply write $\breg{f}{}{\xx}{\bxx}$ (which is, in this case, also known as the Taylor distance).


\section{Estimation with low-complexity penalties} \label{sec:estimlc}

The estimators $\thvar$ and $\thewa$ in \eqref{eq:minP} and \eqref{eq:EWAaggregate} require two essential ingredients: the data loss term $F$ and the prior penalty $J$. We here specify the class of such functions covered in our work, and provide illustrating examples. 

\subsection{Data loss} \label{subsec:fid}

The class of loss functions $F$ that we consider obey the following assumptions:
\begin{enumerate}[label=\textbf{(H.\arabic*)},ref=\bf{(H.\arabic*)}]
\item $F(\cdot,\yy): \RR^n \to \RR$ is $C^1(\RR^n)$ and uniformly convex for all $\yy$ of modulus $\varphi$, i.e.
\[
F(\bsv,\yy) \geq F(\bsu,\yy) + \dotp{\nabla F(\bsu,\yy)}{\bsv-\bsu} + \varphi(\anorm{\bsv-\bsu}{2}),
\]
where $\varphi: \RR_+ \to \RR_+$ is a convex non-decreasing function that vanishes only at $0$.
\label{assump:F}
\item For any $\bxx \in \RR^p$ and $\yy \in \RR^n$, $\int_{\RR^p}\exp\pa{-F(\XX\xx,\yy)/(n\beta)}\abs{\dotp{\nabla F(\XX\xx,\yy)}{\XX(\bxx-\xx)}} d\xx < +\infty$.
\label{assump:Fmom}
\end{enumerate}
Recall that by Lemma~\ref{lem:monconj}, the monotone conjugate $\varphi^+$ of $\varphi$ is a proper, closed, convex, strongly coercive and non-decreasing function on $\RR_+$ that vanishes at $0$. Moreover, $\varphi^{++}=\varphi$. $\varphi^+$ is finite-valued on $\RR_+$ if $\varphi$ is strongly coercive, and it vanishes only at $0$ under e.g.~Lemma~\ref{lem:monconj}(iii). 
  
The class of data loss functions in~\ref{assump:F} is fairly general. It is reminiscent of the negative log-likelihood in the regular exponential family. For the moment assumption~\ref{assump:Fmom} to be satisfied, it is suffient that
\begin{align*}
\int_{\RR^p}\exp\pa{-\varphi\pa{\norm{\XX\xx}{2}}/(n\beta)}\norm{\nabla F(\XX\xx+\bsu^\star),\yy)}{2}\norm{\XX\xx+(\bsu^\star-\XX\bxx)}{2} d\xx < +\infty ,
\end{align*}
where $\bsu^\star$ be a minimizer of $F(\cdot,\yy)$, which is unique by uniform convexity. 
We here provide an example.

\begin{example}\label{ex:F}
Consider the case where\footnote{We consider a scaled version of $\varphi$ for simplicity, but the same conclusions remain valid if we take $\varphi(t)=C t^q/q$, with $C > 0$.} $\varphi(t)=t^q/q$, $q \in ]1,+\infty[$, or equivalently $\varphi^+(t)=t^{q_*}/q_*$ where $1/q+1/q_*=1$. For $q=q_*=2$, \ref{assump:F} amounts to saying that $F(\cdot,\yy)$ is strongly convex for all $\yy$. In particular, \cite[Proposition~10.13]{bauschke2011convex} shows that $F(\bsu,\yy)=\norm{\bsu-\yy}{2}^q/q$ is uniformly convex for $q \in [2,+\infty[$ with modulus $\varphi(t)=C_q t^q/q$, where $C_q > 0$ is a constant that depends solely on $q$.

For \ref{assump:Fmom} to be verified, it is suffient that
\begin{align*}
\int_{\RR^p}\exp\pa{-\norm{\XX\xx}{2}^q/(qn\beta)}\norm{\nabla F(\XX\xx+\bsu^\star),\yy)}{2}\norm{(\XX\xx+\bsu^\star)-\XX\bxx}{2} d\xx < +\infty .
\end{align*}
In particular, taking $F(\bsu,\yy)=\norm{\bsu-\yy}{2}^q/q$, $q \in [2,+\infty[$, we have $\norm{\nabla F(\bsu,\yy)}{2} = \norm{\bsu-\yy}{2}^{q-1}$, and thus~\ref{assump:Fmom} holds since
\begin{align*}
\int_{\RR^p}\exp\pa{-\norm{\XX\xx}{2}^q/(qn\beta)}\norm{\yy - (\XX\xx+\bsu^\star)}{2}^{q-1}\norm{\XX\bxx - (\XX\xx+\bsu^\star)}{2} d\xx < +\infty .
\end{align*}
\end{example}





\subsection{Prior penalty} \label{subsec:prior}

Recall the main definitions and results from convex analysis that are collected in Section~\ref{sec:convana}. Our main assumption on $J$ is the following.
\begin{enumerate}[label=\textbf{(H.\arabic*)},ref=\bf{(H.\arabic*)},start=3]
\item $J: \RR^p \to \RR$ is the gauge of a non-empty convex compact set containing the origin as an interior point.
\label{assump:J}
\end{enumerate}
By Lemma~\ref{lem:convex-gauge}, this assumption is equivalent to saying that $J \eqdef \gamma_{\calC}$ is proper, convex, positively homogeneous, finite-valued and coercive. In turn, $J$ is locally Lipschitz continuous on $\RR^p$. Observe also that by virtue of Lemma~\ref{lem:convex-polar-gauge} and Lemma~\ref{lem:suppcompact}, the polar gauge $J^\circ \eqdef \gamma_{\calC^\circ}$ enjoys the same properties as $J$ in \ref{assump:J}.

\subsection{Decomposability of the prior penalty}\label{subsec:decomp}
We are now in position to provide an important characterization of the subdifferential mapping of a function $J$ satisfying \ref{assump:J}. This characterization will play a pivotal role in our proof of the oracle inequality. 

We start by defining some essential geometrical objects that were introduced in \cite{vaiterimaiai13}.
\begin{definition}[Model Subspace]\label{defn:linmod}
  Let $\xx \in \RR^p$. We denote by $\e{\xx}$ as
  \begin{equation*}
    \e{\xx} = \proj_{\Aff(\partial J(\xx))}(0).
  \end{equation*}
  We denote 
  \begin{equation*}
    \S_{\xx} = \Lin (\partial \J(\xx)) \qandq \T_{\xx} = \S_{\xx}^\bot .
  \end{equation*}
  $\T_{\xx}$ is coined the \emph{model subspace} of $\xx$ associated to $J$.
\end{definition}
It can be shown, see \cite[Proposition~5]{vaiterimaiai13}, that $\xx \in \T_{\xx}$, hence the name model subspace. When $J$ is differentiable at $\xx$, we have $\e{\xx}=\nabla J(\xx)$ and $\T_{\xx}=\RR^p$. When $J$ is the $\ell_1$-norm (Lasso), the vector $\e{\xx}$ is nothing but the sign of $\xx$. Thus, $\e{\xx}$ can be viewed as a generalization of the sign vector. Observe also that $\e{\xx}=\proj_{\T_{\xx}}(\partial J(\xx))$, and thus $\e{\xx} \in \T_{\xx} \cap \Aff(\partial J(\xx))$. However, in general, $\e{\xx} \not\in \partial J(\xx)$.

We now provide a fundamental equivalent description of the subdifferential of $J$ at $\xx$ in terms of $\e{\xx}$, $\T_{\xx}$, $\S_{\xx}$ and the polar gauge $J^\circ$.
\begin{theorem}\label{thm:decomp}
Let $J$ satisfy \ref{assump:J}. Let $\xx \in \RR^p$ and $\f{\xx} \in \ri(\partial J(\xx))$.
\begin{enumerate}[label=(\roman*)]
\item The subdifferential of $J$ at $\xx$ reads
  \begin{align*}
    \partial J(\xx) &= \Aff(\partial J(\xx)) \cap \calC^\circ \\
    &=
    \enscond{\bseta \in \RR^n}
    {
      \bseta_{\T_{\xx}} = \e{\xx}
      \qandq
      \inf_{\tau \geq 0}\max\pa{J^\circ\pa{\tau\e{\xx} + \bseta_{\S_{\xx}} + (\tau-1)\proj_{\S_{\xx}}\f{\xx}},\tau} \leq 1
    } .
  \end{align*}
\item For any $\bsom \in \RR^p$, $\exists \bseta \in \partial J(\xx)$ such that
\[
J(\bsom_{\S_{\xx}}) = \dotp{\bseta_{\S_{\xx}}}{\bsom_{\S_{\xx}}} .
\]
\end{enumerate}
\end{theorem}

\begin{proof}
\begin{enumerate}[label=(\roman*)]
\item This follows by piecing together \cite[Theorem~1, Proposition~4 and Proposition~5(iii)]{vaiterimaiai13}.
\item From \cite[Proposition~5(iv)]{vaiterimaiai13}, we have
\[
\sigma_{\partial J(\xx) - \f{\xx}}(\bsom) = J(\bsom_{\S_{\xx}}) - \dotp{\proj_{\S_{\xx}}\f{\xx}}{\bsom_{\S_{\xx}}} .
\]
Thus there exists a supporting point $\bsv \in \partial J(\xx) - \f{\xx} \subset \S_{\xx}$ with normal vector $\bsom$ \cite[Corollary~7.6(iii)]{bauschke2011convex}, i.e.
\[
\sigma_{\partial J(\xx) - \f{\xx}}(\bsom) = \dotp{\bsv}{\bsom_{\S_{\xx}}} .
\]
Taking $\bseta=\bsv+\f{\xx}$ concludes the proof.
\end{enumerate}
\end{proof}

\begin{remark}
The coercivity assumption in \ref{assump:J} is not needed for Theorem~\ref{thm:decomp} to hold.
\end{remark}

The decomposability of described in Theorem~\ref{thm:decomp}(i) depends on the particular choice of the mapping $\xx \mapsto \f{\xx} \in \ri(\partial J(\xx))$. An interesting situation is encountered when $\e{\xx} \in \ri(J(\xx))$, so that one can choose $\f{\xx}=\e{\xx}$. Strong gauges, see \cite[Definition~6]{vaiterimaiai13}, are precisely a class of gauges for which this situation occurs, and in this case, Theorem~\ref{thm:decomp}(i) has the simpler form
\begin{equation}\label{eq:decompstrong}
\partial J(\xx) = \Aff(\partial J(\xx)) \cap \calC^\circ =
\enscond{\bseta \in \RR^n}
{
  \bseta_{\T_{\xx}} = \e{\xx}
  \qandq
  J^\circ(\bseta_{\S_{\xx}}) \leq 1
} .
\end{equation}

The Lasso, group Lasso and nuclear norms are typical examples of (symmetric) strong gauges. However, analysis sparsity penalties (e.g. the fused Lasso) or the $\ell_\infty$-penalty are not strong gauges, though they obviously satisfy \ref{assump:J}. See the next section for a detailed discussion.

\subsection{Calculus with the prior family}\label{subsec:stable}
The family of penalties complying with~\ref{assump:J} form a robust class enjoying important calculus rules. In particular it is closed under the sum and composition with an injective linear operator as we now prove.

\begin{lemma}\label{lem:stable}
The set of functions satisfying \ref{assump:J} is closed under addition\footnote{It is obvious that the same holds with any positive linear combination.} and pre-composition by an injective linear operator. More precisely, the following holds:
\begin{enumerate}[label=(\roman*)]
\item Let $J$ and $G$ be two gauges satisfying \ref{assump:J}. Then $H \eqdef J+G$ also obeys \ref{assump:J}. Moreover, 
\begin{enumerate}[label=(\alph*)]
\item $\T^H_{\xx} = \T^J_{\xx} \cap \T^G_{\xx}$ and $\e{\xx}^H = \proj_{\T^H_{\xx}}(\e{\xx}^J+\e{\xx}^G)$, where $\T^J_{\xx}$ and $\e{\xx}^J$ (resp. $\T^G_{\xx}$ and $\e{\xx}^G$) are the model subspace and vector at $\xx$ associated to $J$ (resp. $G$);
\item $H^\circ(\bsom)=\max_{\rho \in [0,1]}\co{\inf\pa{\rho J^\circ(\bsom),(1-\rho)G^\circ(\bsom)}}$.
\end{enumerate}
\item Let $J$ be a gauge satisfying \ref{assump:J}, and $\bsD: \RR^q \to \RR^p$ be surjective. Then $H \eqdef J \circ \bsD^\top$ also fulfills \ref{assump:J}. Moreover,
\begin{enumerate}[label=(\alph*)]
\item $\T^H_{\xx} = \Ker(\bsD_{\S^J_{\bsu}}^\top)$ and $\e{\xx}^H = \proj_{\T^H_{\xx}}\bsD\e{\bsu}^J$, where $\T^J_{\bsu}$ and $\e{\bsu}^J$ are the model subspace and vector at $\bsu \eqdef \bsD^\top\xx$ associated to $J$;
\item $H^\circ(\bsom)=J^\circ(\bsD^+\bsom)$, where $\bsD^+=\bsD^\top\bpa{\bsD\bsD^\top}^{-1}$.
\end{enumerate}
\end{enumerate}
\end{lemma}

The outcome of Lemma~\ref{lem:stable} is naturally expected. For instance, assertion (i) states that combining several penalties/priors will promote objects living on the intersection of the respective low-complexity models. Similarly, for (ii), one promotes low-complexity in the image of the analysis operator $\bsD^\top$. It then follows that one has not to deploy an ad hoc analysis when linearly pre-composing or combining (or both) several penalties (e.g. $\ell_1$+nuclear norms for recovering sparse and low-rank matrices) since our unified analysis in Section~\ref{sec:soi} will apply to them just as well.

\begin{proof}
\begin{enumerate}[label=(\roman*)]
\item Convexity, positive homogeneity, coercivity and finite-valuedness are straightforward.
\begin{enumerate}[label=(\alph*)]
\item This is \cite[Proposition~8(i)-(ii)]{vaiterimaiai13}.
\item We have from Lemma~\ref{lem:convex-polar-gauge} and calculus rules on support functions, 
\begin{align*}
H^\circ(\bsom)	&= \sigma_{J(\xx)+G(\xx) \leq 1}(\bsom) = \sup_{J(\xx)+G(\xx) \leq 1} \dotp{\bsom}{\xx} = \max_{\rho \in [0,1]} \sup_{J(\xx) \leq \rho,G(\xx) \leq 1-\rho} \dotp{\bsom}{\xx} \\
\text{\scriptsize{(\cite[Theorem~V.3.3.3]{hiriart1996convex})}} &= \max_{\rho \in [0,1]} \co{\inf\pa{\sigma_{J(\xx) \leq \rho}(\bsom),\sigma_{G(\xx) \leq 1-\rho}(\bsom)}} \\
\text{\scriptsize{(Positive homogeneity)}}			&= \max_{\rho \in [0,1]} \co{\inf\pa{\rho\sigma_{J(\xx) \leq 1}(\bsom),(1-\rho)\sigma_{G(\xx) \leq 1}(\bsom)}} \\
\text{\scriptsize{(Lemma~\ref{lem:convex-polar-gauge})}}	&= \max_{\rho \in [0,1]} \co{\inf\pa{\rho J^\circ(\bsom),(1-\rho)G^\circ(\bsom)}} .
\end{align*}
\end{enumerate}
\item Again, Convexity, positive homogeneity and finite-valuedness are immediate. Coercivity holds by injectivity of $\bsD^\top$.
\begin{enumerate}[label=(\alph*)]
\item This is \cite[Proposition~10(i)-(ii)]{vaiterimaiai13}.
\item Denote $J = \gamma_{\calC}$. We have
\begin{align*}
H^\circ(\bsom)	&= \sup_{\bsD^\top\xx \in \calC} \dotp{\bsom}{\xx} \\
\text{\scriptsize{($\bsD^\top$ is injective)}}			&= \sup_{\bsD^\top\xx \in \calC} \dotp{\bsD^+\bsom}{\bsD^\top\xx} \\
								&= \sup_{\bsu \in \calC \cap \Span(\bsD^\top)} \dotp{\bsD^+\bsom}{\bsu} \\
\text{\scriptsize{(\cite[Theorem~V.3.3.3]{hiriart1996convex} and Lemma~\ref{lem:convex-polar-gauge})}}	
								&= \co{\inf\pa{J^\circ(\bsD^+\bsom),\iota_{\Ker(\bsD)}(\bsD^+\bsom)}} \\
								&= J^\circ(\bsD^+\bsom) .
\end{align*}
where in the last equality, we used the fact that $\bsD^+\bsom \in \Span\bpa{\bsD^\top}=\Ker(\bsD)^\perp$, and thus $\iota_{\Ker(\bsD)}(\bsD^+\bsom)=+\infty$ unless $\bsom=0$, and $J^\circ$ is continuous and convex by \ref{assump:J} and Lemma~\ref{lem:convex-polar-gauge}.
\end{enumerate}
\end{enumerate}
\end{proof}

\subsection{Examples}\label{subsec:decompexamp}

\subsubsection{Lasso}\label{subsec:decomplasso}
The Lasso regularization is used to promote the sparsity of the minimizers, see~\cite{BuhlmannVandeGeerBook11} for a comphensive review. It corresponds to choosing $J$ as the $\ell_1$-norm
\begin{equation}
  \label{eq:Jlasso}
  J(\xx) = \norm{\xx}{1} = \sum_{i=1}^p \abs{\xx_i} .
\end{equation}
It is also referred to as $\ell_1$-synthesis in the signal processing community, in contrast to the more general $\ell_1$-analysis sparsity penalty detailed below.

We denote $(\bs a_i)_{1 \leq i \leq p}$ the canonical basis of $\RR^p$ and $\supp(\xx) \eqdef \enscond{i \in \ens{1,\dots,p}}{\xx_i \neq 0}$. Then,
\begin{equation}\label{eq:exlasso}
\T_\xx = \Span \ens{(\bs a_i)_{i \in \supp(\xx)}}, \quad 
(\e{\xx})_i = 
\begin{cases} 
\sign(\xx_i) 	& \text{if~} i \in \supp(\xx) \\
0		& \text{otherwise}
\end{cases},
\qandq J^\circ=\norm{\cdot}{\infty} .
\end{equation}

\subsubsection{Group Lasso}\label{subsec:decompglasso}
The group Lasso has been advocated to promote sparsity by groups, i.e. it drives all the coefficients in one group to zero together hence leading to group selection, see~\cite{bakin1999adaptive,yuan2006model,bach2008consistency,Wei10} to cite a few. The group Lasso penalty with $L$ groups reads
\begin{equation}
  \label{eq:Jglasso}
  J(\xx) = \norm{\xx}{1,2} \eqdef \sum_{i=1}^L \norm{\xx_{b_i}}{2} .
\end{equation}
where $\bigcup_{i=1}^L b_i = \ens{1,\ldots,p}$, $b_i, b_j \subset \ens{1,\ldots,p},$ and  $b_i \cap b_j = \emptyset$ whenever $i \neq j$. Define the group support as $\bsupp(\xx) \eqdef \enscond{i \in \ens{1,\ldots,L}}{\xx_{b_i} \neq 0}$.
Thus, one has
\begin{equation}\label{eq:exglasso}
\T_\xx = \Span \ens{(a_j)_{\enscond{j}{\exists i \in \bsupp(\xx), j \in b_i}}}, 
(\e{\xx})_{b_i} = 
\begin{cases} 
\tfrac{\xx_{b_i}}{\anorm{\xx_{b_i}}{2}} & \text{if~} i \in \bsupp(\xx) \\
0 				& \text{otherwise}
\end{cases}, \tandt
J^\circ(\bsom) = \max_{i \in \ens{1,\ldots,L}} \anorm{\bsom_{b_i}}{2} .
\end{equation}

\subsubsection{Analysis (group) Lasso}\label{subsec:decompanaglasso}
One can push the structured sparsity idea one step further by promoting group/block sparsity through a linear operator, i.e. analysis-type sparsity. Given a linear operator $\bsD: \RR^q \to \RR^p$ (seen as a matrix), the analysis group sparsity penalty is 
\begin{equation}
  \label{eq:Janaglasso}
  J(\xx) = \norm{\bsD^\top \xx}{1,2} .
\end{equation}
This encompasses the 2-D isotropic total variation~\cite{rudin1992nonlinear}. For when all groups of cardinality one, we have the analysis-$\ell_1$ penalty (a.k.a. general Lasso), which encapsulates several important penalties including that of the 1-D total variation~\cite{rudin1992nonlinear}, and the fused Lasso \cite{tibshirani2005sparsity}. The overlapping group Lasso \cite{jacob-overlap-synthesis} is also a special case of \eqref{eq:Jglasso} by taking $\bsD^\top$ to be an operator that exactract the blocks~\cite{peyre2011adaptive,chen-proximal-overlap} (in which case $\bsD$ has even orthogonal rows). 

Let $\Lambda_{\xx}=\bigcup_{i \in \bsupp(\bsD^\top \xx)} b_i$ and $\Lambda_{\xx}^c$ its complement. From Lemma~\ref{lem:stable}(ii) and \eqref{eq:exglasso}, we get
\begin{equation}\label{eq:exanaglasso}
\T_\xx = \Ker(\bsD_{\Lambda_{\xx}^c}^\top), \quad \e{\xx} = \proj_{\T_{\xx}}\bsD\e{\bsD^\top \xx}^{\anorm{}{1,2}} \qwhereq 
\pa{\e{\bsD^\top \xx}^{\anorm{}{1,2}}}_{b_i} = 
\begin{cases} 
\tfrac{\pa{\bsD^\top\xx}_{b_i}}{\anorm{\pa{\bsD^\top\xx}_{b_i}}{2}} 	& \text{if~} i \in \bsupp(\bsD^\top \xx) \\
0 									& \text{otherwise} .
\end{cases}
\end{equation}
If, in addition, $\bsD$ is surjective, then by virtue of Lemma~\ref{lem:stable}(ii) we also have
\begin{equation}\label{eq:exanaglassoJo}
J^\circ(\bsom) = \anorm{\bsD^+\bsom}{\infty,2} \eqdef \max_{i \in \ens{1,\ldots,L}} \anorm{(\bsD^+\bsom)_{b_i}}{2} 
\end{equation}

\subsubsection{Anti-sparsity}\label{subsec:decomplinf}
If the vector to be estimated is expected to be flat (anti-sparse), this can be captured using the $\ell_\infty$ norm (a.k.a. Tchebychev norm) as prior
\begin{equation}
  \label{eq:Jlinf}
      J(\xx) = \norm{\xx}{\infty} = \max_{i \in \ens{1,\dots,p}} \abs{\xx_i}.
\end{equation}
The $\ell_\infty$ regularization has found applications in several fields \cite{jegou2012anti,lyubarskii2010uncertainty,studer12signal}.
Suppose that $\xx \neq 0$, and define the saturation support of $\xx$ as $\Isat_{\xx} \eqdef \enscond{i \in \ens{1,\dots,p}}{\abs{\xx_i}=\anorm{\xx}{\infty}} \neq \emptyset$. From \cite[Proposition~14]{vaiterimaiai13}, we have
\begin{equation}\label{eq:exlinf}
\T_\xx = \enscond{\bxx \in \RR^p}{\bxx_{\Isat_{\xx}} \in \RR \sign(\xx_{\Isat_{\xx}})}, \quad 
(\e{\xx})_i=
\begin{cases} 
\sign(\xx_i)/|\Isat_{\xx}|	& \text{if~} i \in \Isat_{\xx} \\
0				& \text{otherwise}
\end{cases}, \qandq
J^\circ = \anorm{\cdot}{1} .
\end{equation}

\subsubsection{Nuclear norm}\label{subsec:decompnuc}
The natural extension of low-complexity priors to matrices $\xx \in \RR^{p_1 \times p_2}$ is to penalize the singular values of the matrix. Let $\rank(\xx)=r$, and $\xx = \bs U \diag(\uplambda(\xx))\bs V^\top$ be a reduced rank-$r$ SVD decomposition, where $\bs U \in \RR^{p_1 \times r}$ and $\bs V \in \RR^{p_2 \times r}$ have orthonormal columns, and $\uplambda(\xx) \in (\RR_+ \setminus \ens{0})^{r}$ is the vector of singular values $(\uplambda_1(\xx),\cdots,\uplambda_r(\xx))$ in non-increasing order. The nuclear norm of $\xx$ is
\begin{equation}\label{eqjJnuc} 
  J(\xx) = \norm{\xx}{*} = \norm{\uplambda(\xx)}{1} .
\end{equation}
This penalty is the best convex surrogate to enforce a low-rank prior. It has been widely used for various applications \cite{recht2010guaranteed,candes2009exact,CandesRPCA11,fazel2001rank,CandesPhaseLift}. 
 
Following e.g. \cite[Example~21]{vaiteraism16}, we have
\begin{equation}\label{eq:exnuc}
\T_{\xx} = \enscond{\bs U \bs A^\top + \bs B \bs V^\top}{ \bs A \in \RR^{p_2 \times r}, \bs B \in \RR^{p_1 \times r}  }, \quad \e{\xx} = \bs U \bs V^\top \qandq J^\circ(\bsom) = \normOP{\bsom}{2}{2} = \norm{\uplambda(\bsom)}{\infty} .
\end{equation}

\section{Oracle inequalities for a general loss} \label{sec:soi}

Before delving into the details, in the sequel, we will need a bit of notations.

We recall $\T_{\xx}$ and $\e{\xx}$ the model subspace and vector associated to $\xx$ (see Definition~\ref{defn:linmod}). Denote $\S_{\xx}=\T_{\xx}^\perp$. Given two coercive finite-valued gauges $J_1 = \gamma_{\calC_1}$ and $J_2 = \gamma_{\calC_2}$, and a linear operator $\bs A$, we define $\normOP{\bs A}{J_1}{J_2}$ the \emph{operator bound} as
\begin{equation*}
  \normOP{\bs A}{J_1}{J_2} = \sup_{\xx \in \calC_1} J_2(\bs A \xx) .
\end{equation*}
Note that $\normOP{\bs A}{J_1}{J_2}$ is bounded (this follows from Lemma~\ref{lem:convex-gauge}(v)). Furthermore, we have from Lemma~\ref{lem:convex-polar-gauge} that 
\begin{equation*}
  \normOP{\bs A}{J_1}{J_2} = \sup_{\xx \in \calC_1} \sup_{\bsom \in \calC^\circ_2} \dotp{\bs A^\top \bsom}{\xx} = \sup_{\bsom \in \calC^\circ_2}\sup_{\xx \in \calC_1}\dotp{\bs A^\top \bsom}{\xx} = \sup_{\bsom \in \calC^\circ_2} J_1^\circ(\bs A^\top \bsom) = \anormOP{\bs A^\top}{J^\circ_2}{J^\circ_1} .
\end{equation*}
In the following, whenever it is clear from the context, to lighten notation when $J_i$ is a norm, we write the subscript of the norm instead of $J_i$ (e.g. $p$ for the $\ell_p$ norm, $*$ for the nuclear norm, etc.).\\

Our main result will involve a measure of well-conditionedness of the design matrix $\XX$ when restricted to some subspace $\T$. More precisely, for $c > 0$, we introduce the coefficient
\begin{equation}\label{eq:cfac}
\cfac{T}{c} = \inf_{\enscond{\bsom \in \RR^p}{J(\bsom_\S) < c J(\bsom_\T)}} \frac{\normOP{\proj_{\T}}{2}{J}\norm{\XX\bsom}{2}}{n^{1/2}\pa{J(\bsom_\T) - J(\bsom_\S)/c}} .
\end{equation}
This generalizes the compatibility factor introduced in \cite{buhlmannvandegeer09} for the Lasso (and used in \cite{dalalyan16}). The experienced reader may have recognized that this factor is reminescent of the null space property and restricted injectivity that play a central role in the analysis of the performance guarantees of variational/penalized estimators \eqref{eq:minP}; see \cite{Fadilidecomposable13,vaiterimaiai13,vaiterps14,vaiteraism16,vaiterbookchap15}. One can see in particular that $\cfac{T}{c}$ is larger than the smallest singular value of $\XX_\T$.\\

The oracle inequalites will provided in terms of the loss
\[
R_n\bpa{\xx,\xx_0} = \tfrac{1}{n}\breg{F}{}{\XX\xx}{\XX\xx_0} .
\]

\subsection{Oracle inequality for $\thewa$}\label{subsec:SOIEWA}

We are now ready to establish our first main result: an oracle inequality for the EWA estimator \eqref{eq:EWAaggregate}.

\begin{theorem}\label{thm:SOIEWA}
Consider the EWA estimator $\thewa$ in \eqref{eq:EWAaggregate} with the density \eqref{eq:EWAaggregator}, where $F$ and $J$ satisfy Assumptions~\ref{assump:F}-\ref{assump:Fmom} and \ref{assump:J}. Then, for any $\tau > 1$ such that $\lambda_n \geq \tau J^\circ\pa{-\XX^\top\nabla F(\XX\xx_0,\yy)}/n$, the following holds,
\begin{equation}\label{eq:SOIEWA}
R_n\bpa{\thewa,\xx_0} \leq \inf_{\xx\in\RR^p} \pa{R_n\bpa{\xx,\xx_0} + \frac{1}{n}\varphi^+\pa{\frac{\lambda_n\sqrt{n}\bpa{\tau J^\circ(\e{\xx})+1}\normOP{\proj_{\T_{\xx}}}{2}{J}}{\tau\cfac{\T_{\xx}}{\tfrac{\tau J^\circ(\e{\xx})+1}{\tau-1}}}}} + p\beta .
\end{equation}
\end{theorem}

\begin{remark}{~}\label{rem:SOIEWA}\\\vspace*{-0.5cm}
\begin{enumerate}[label=\arabic*.,ref=\arabic*]
\item It should be emphasized that Theorem~\ref{thm:SOIEWA} is actually a deterministic statement for a fixed choice of $\lambda_n$. Probabilistic analysis will be required when the result is applied to particular statistical models as we will see later. For this, we will use concentration inequalities in order to provide bounds that hold with high probability over the data.
\item The oracle inequality is sharp. The remainder in it has two terms. The first one encodes the complexity of the model promoted by $J$. The second one, $p\beta$, captures the influence of the temperature parameter. In particular, taking $\beta$ sufficiently small of the order $O\pa{(pn)^{-1}}$, this term becomes $O(n^{-1})$. 
\item When $\varphi(t)=\sigm t^2/2$, i.e. $F(\cdot,\yy)$ is $\sigm$-strongly convex, then $\varphi^+(t)=t^2/(2\nu)$, and the reminder term becomes
\begin{equation}\label{eq:remquad}
\frac{\lambda_n^2\bpa{\tau J^\circ(\e{\xx})+1}^2\normOP{\proj_{\T_{\xx}}}{2}{J}^2}{2\tau^2\nu\cfac{\T_{\xx}}{\tfrac{\tau J^\circ(\e{\xx})+1}{\tau-1}}^2} .
\end{equation}
If, moreover, $\nabla F$ is also $\sigM$-Lipschitz continuous, then it can be shown that $R_n\bpa{\xx,\xx_0}$ is equivalent to a quadratic loss. This means that the oracle inequality in Theorem~\ref{thm:SOIEWA} can be stated in terms of the quadratic prediction error. However, the inequality is not anymore sharp in this case as a constant factor equal to the condition number $\sigM/\sigm \geq 1$ naturally multiplies the right-hand side.  
\label{rem:SOIEWA1}

\item If $J$ is such that $\e{\xx} \in \partial J(\xx) \subset \calC^\circ$ (typically for a strong gauge by \eqref{eq:decompstrong}), then $J^\circ(\e{\xx}) \leq 1$ (in fact an equality if $\xx \neq 0$). Thus the term $J^\circ(\e{\xx})$ can be omitted in \eqref{eq:SOIEWA}.
\label{rem:SOIEWA2}

\item A close inspection of the proof of Theorem~\ref{thm:SOIEWA} reveals that the term $p\beta$ can be improved to the smaller bound
\[
p\beta + \pa{V_n(\thewa)-\EE{\mu_n}{V_n(\xx)}} ,
\]
where the upper-bound is a consequence of Jensen inequality.
\label{rem:SOIEWA3}
\end{enumerate}
\end{remark}

\begin{proof}
By convexity of $J$ and assumption~\ref{assump:F}, we have for any $\bseta \in \partial V_n(\xx)$ and any $\bxx \in \RR^p$,
\[
\breg{V_n}{\bseta}{\bxx}{\xx} \geq \frac{1}{n}\varphi\bpa{\norm{\XX\bxx-\XX\xx}{2}} .
\]
Since $\varphi$ is non-decreasing and convex, $\varphi \circ \anorm{\cdot}{2}$ is a convex function. Thus, taking the expectation w.r.t. to $\mu_n$ on both sides and using Jensen inequality, we get
\begin{align*}
V_n(\bxx)  &\geq \EE{\mu_n}{V_n(\xx)} + \EE{\mu_n}{\dotp{\bseta}{\bxx-\xx}} + \frac{1}{n}\EE{\mu_n}{\varphi\bpa{\norm{\XX\bxx-\XX\xx}{2}}} \\
	   &\geq V_n(\thewa) + \EE{\mu_n}{\dotp{\bseta}{\bxx-\xx}} + \frac{1}{n}\varphi\bpa{\norm{\XX\bxx-\XX\thewa}{2}} .
\end{align*}
This holds for any $\bseta \in \partial V_n(\xx)$, and in particular at the minimal selection $\mnsel{\partial V_n(\xx)}$ (see Section~\ref{sec:avdotp} for details). It then follows from the pillar result in Proposition~\ref{prop:avdotp}\footnote{In the appendix, we provide a self-contained proof based on a novel Moreau-Yosida regularization argument. In \cite[Corollary~1 and 2]{dalalyan16}, an alternative proof is given using an absolute continuity argument since $\mu_n$ is locally Lipschitz, hence a Sobolev function. 
} that
\[
\EE{\mu_n}{\dotp{\mnsel{\partial V_n(\xx)}}{\bxx-\xx}} = -p\beta .
\]

We thus deduce the inequality
\begin{equation}\label{eq:Vbndewa}
V_n(\thewa) - V_n(\bxx)  \leq p\beta - \frac{1}{n}\varphi\bpa{\norm{\XX\thewa-\XX\bxx}{2}} , \quad \forall \bxx \in \RR^p .
\end{equation}
By definition of the Bregman divergence, we have
\begin{align*}
R_n\bpa{\thewa,\xx_0} - R_n\bpa{\xx,\xx_0} &=\frac{1}{n}\Bpa{F(\XX\thewa,\yy) - F(\XX\xx,\yy) + \dotp{-\XX^\top\nabla F(\XX\xx_0,\yy))}{\thewa-\xx}} \\
					   &=\pa{V_n(\thewa) - V_n(\xx)} + \frac{1}{n}\dotp{-\XX^\top\nabla F(\XX\xx_0,\yy)}{\thewa-\xx} \\
					   & \qquad - \lambda_n\bpa{J(\thewa)-J(\xx)} .
\end{align*}
By virtue of the duality inequality \eqref{eq:dualineq}, we have
\begin{align*}
R_n\bpa{\thewa,\xx_0} - R_n\bpa{\xx,\xx_0} &\leq \pa{V_n(\thewa) - V_n(\xx)} + \frac{1}{n}J^\circ\pa{-\XX^\top\nabla F(\XX\xx_0,\yy)}J(\thewa-\xx) \\
					   & \qquad - \lambda_n\bpa{J(\thewa)-J(\xx)} \\
					   &\leq \pa{V_n(\thewa) - V_n(\xx)} + \frac{\lambda_n}{\tau}\pa{J(\thewa-\xx) - \tau\bpa{J(\thewa)-J(\xx)}} .
\end{align*}
Denote $\bsom=\thewa-\xx$. By virtue of \ref{assump:J}, Theorem~\ref{thm:decomp} and \eqref{eq:dualineq}, we obtain
\begin{align*}
J(\bsom) - \tau\bpa{J(\thewa)-J(\xx)} 	&\leq J(\bsom_{\T_{\xx}}) + J(\bsom_{\S_{\xx}}) - \tau \dotp{\e{\xx}}{\bsom_{\T_{\xx}}} - \tau J(\bsom_{\S_{\xx}}) \\
					&\leq J(\bsom_{\T_{\xx}}) + J(\bsom_{\S_{\xx}}) + \tau J^\circ(\e{\xx})J(\bsom_{\T_{\xx}}) - \tau J(\bsom_{\S_{\xx}}) \\
					&= \bpa{\tau J^\circ(\e{\xx})+1}J(\bsom_{\T_{\xx}}) - (\tau-1)J(\bsom_{\S_{\xx}}) \\
					&\leq \bpa{\tau J^\circ(\e{\xx})+1}\pa{J(\bsom_{\T_{\xx}}) - \tfrac{\tau-1}{\tau J^\circ(\e{\xx})+1}J(\bsom_{\S_{\xx}})} .
\end{align*}
This inequality together with \eqref{eq:Vbndewa} (applied with $\bxx=\xx$) and \eqref{eq:cfac} yield
\begin{align*}
R_n\bpa{\thewa,\xx_0} - R_n\bpa{\xx,\xx_0} &\leq p\beta - \frac{1}{n}\varphi\pa{\norm{\XX\bsom}{2}} + \frac{\lambda_n\bpa{\tau J^\circ(\e{\xx})+1}\normOP{\proj_{\T_{\xx}}}{2}{J}\norm{\XX\bsom}{2}}{n^{1/2}\tau\cfac{\T_{\xx}}{\tfrac{\tau J^\circ(\e{\xx})+1}{\tau-1}}} \\
					   &\leq p\beta + \frac{1}{n}\varphi^+\pa{\frac{\lambda_n\sqrt{n}\bpa{\tau J^\circ(\e{\xx})+1}\normOP{\proj_{\T_{\xx}}}{2}{J}}{\tau\cfac{\T_{\xx}}{\tfrac{\tau J^\circ(\e{\xx})+1}{\tau-1}}}},
\end{align*}
where we applied Fenchel-Young inequality \eqref{eq:fenineq} to get the last bound. Taking the infimum over $\xx \in \RR^p$ yields the desired result.
\end{proof}

\paragraph{Stratifiable functions} Theorem~\ref{thm:SOIEWA} has a nice instanciation when $\RR^p$ can be partitioned into a collection of subsets $\ens{\calM_i}_i$ that form a stratification of $\RR^p$. That is, $\RR^p$ is a finite disjoint union $\cup_i \calM_i$ such that the partitioning sets $\calM_i$ (called strata) must fit nicely together and the stratification is endowed with a partial ordering for the closure operation. For example, it is known that a polyhedral function has a polyhedral stratification, and more generally, semialgebraic functions induce stratifications into finite disjoint unions of manifolds; see, e.g., \cite{coste2002intro}. Another example is that of partly smooth convex functions thoroughly studied in \cite{vaiterimaiai13,vaiterps14,vaiteraism16,vaiterbookchap15} for various statistical and inverse problems. These functions induce a stratification into strata that are $C^2$-smooth submanifolds of $\RR^p$. In turns out that all popular penalty functions discussed in this paper are partly smooth (see \cite{vaiteraism16,vaiterbookchap15}).
Let's denote $\Mscr$ the set of strata associated to $J$. With this notation at hand, the oracle inequality \eqref{eq:SOIEWA} now reads
\begin{equation}\label{eq:SOIEWAps}
R_n\bpa{\thewa,\xx_0} \leq \inf_{\substack{\calM \in \Mscr\\\xx\in\calM}} \pa{R_n\bpa{\xx,\xx_0} + \frac{1}{n}\varphi^+\pa{\frac{\lambda_n\sqrt{n}\bpa{\tau J^\circ(\e{\xx})+1}\normOP{\proj_{\T_{\xx}}}{2}{J}}{\tau\cfac{\T_{\xx}}{\tfrac{\tau J^\circ(\e{\xx})+1}{\tau-1}}}}} + p\beta .
\end{equation}


\subsection{Oracle inequality for $\thvar$}\label{subsec:SOIPEN}

The next result establishes that $\thvar$ satisfies a sharp prediction oracle inequality that we will compare to \eqref{eq:SOIEWA}.

\begin{theorem}\label{thm:SOIPEN}
Consider the penalized estimator $\thvar$ in \eqref{eq:minP}, where $F$ and $J$ satisfy Assumptions~\ref{assump:F} and \ref{assump:J}. Then, for any $\tau > 1$ such that $\lambda_n \geq \tau J^\circ\pa{-\XX^\top\nabla F(\XX\xx_0,\yy)}/n$, the following holds,
\begin{equation}\label{eq:SOIPEN}
R_n\bpa{\thvar,\xx_0} \leq \inf_{\xx\in\RR^p} \pa{R_n\bpa{\xx,\xx_0} + \frac{1}{n}\varphi^+\pa{\frac{\lambda_n\sqrt{n}\bpa{\tau J^\circ(\e{\xx})+1}\normOP{\proj_{\T_{\xx}}}{2}{J}}{\tau\cfac{\T_{\xx}}{\tfrac{\tau J^\circ(\e{\xx})+1}{\tau-1}}}}} .
\end{equation}
\end{theorem}

\begin{proof}
The proof follows the same lines as that of Theorem~\ref{thm:SOIEWA} except that we use the fact that $\thvar$ is a global minimizer of $V_n$, i.e. $0 \in \partial V_n(\thvar)$. Indeed, we have for any $\xx \in \RR^p$
\begin{align}\label{eq:Vbndpen}
V_n(\xx)  &\geq V_n(\thvar) + \frac{1}{n}\varphi\bpa{\norm{\XX\xx-\XX\thvar}{2}} .
\end{align}
Continuing exactly as just after \eqref{eq:Vbndewa}, replacing $\thewa$ with $\thvar$ and invoking \eqref{eq:Vbndpen} instead of \eqref{eq:Vbndewa}, we arrive at the claimed result.
\end{proof}

\begin{remark}{~}\label{rem:SOIPEN}\\\vspace*{-0.5cm}
\begin{enumerate}[label=\arabic*.,ref=\arabic*]
\item Observe that the penalized estimator $\thvar$ does not require the moment assumption~\ref{assump:Fmom} for \eqref{eq:SOIPEN} to hold. The convexity assumption on $\varphi$ in~\ref{assump:F}, which was important to apply Jensen's inequality in the proof of~\eqref{eq:SOIEWA}, is not needed either to get \eqref{eq:SOIPEN}.

\item As we remarked for Theorem~\ref{thm:SOIEWA}, Theorem~\ref{thm:SOIPEN} is also a deterministic statement for a fixed choice of $\lambda_n$ that holds for any minimizer $\thvar$, which is not unique in general. The condition on $\lambda_n$ is similar to the one in \cite{negahban2010unified} where authors established different guarantees for $\thvar$.

\end{enumerate}
\end{remark}

One clearly sees that the difference between the prediction performance of $\thewa$ and $\thvar$ lies in the term $p\beta$ (or rather its lower-bound in Remark~\ref{rem:SOIEWA}-\ref{rem:SOIEWA3}). Thus letting $\beta \to 0$ in \eqref{eq:SOIEWA}, one recovers the oracle inequality \eqref{eq:SOIPEN} of penalized estimators. In particular, for $\beta=O\pa{(pn)^{-1}}$, this is on the order $O(n^{-1})$.

\subsection{Oracle inequalities in probability}\label{subsec:SOIprob}
It remains to check when the event $\Escr = \ens{\lambda_n \geq \tau J^\circ\pa{-\XX^\top\nabla F(\XX\xx_0,\yy)}/n}$ holds with high probability when $\yy$ is random. We will use concentration inequalities in order to provide bounds that hold with high probability over the data. 
Toward this goal, we will need the following assumption.

\begin{enumerate}[label=\textbf{(H.\arabic*)},ref=\bf{(H.\arabic*)},start=4]
\item $\yy=(\yy_1,\yy_2,\cdots,\yy_n)$ are independent and identically distributed observations, and \linebreak$F(\bsu,\yy)=\sum_{i=1}^n f_i(\bsu_i,\yy_i)$, $f_i: \RR \times \RR \to \RR$. Moreover, 
\begin{enumerate}[label=(\roman*)]
\item $\EE{}{\abs{f_i((\XX\xx_0)_i,\yy_i)}} < +\infty$, $\forall 1 \leq i \leq n$ ;
\item $\abs{f_i'((\XX\xx_0)_i,t)} \leq g(t)$, where $\EE{}{g(\yy_i)} < +\infty$, $\forall 1 \leq i \leq n$;
\item Bernstein moment condition: $\forall 1 \leq i \leq n$ and all integers $m \geq 2$, $\EE{}{\abs{f_i'((\XX\xx_0)_i,\yy_i)}^m} \leq m! \kappa^{m-2} \sigma_i^2/2$ for some constants $\kappa > 0$, $\sigma_i > 0$ independent of $n$.
\end{enumerate}
\label{assump:ygenF}
\end{enumerate}

Observe that under \ref{assump:ygenF}, and by virtue of Lemma~\ref{lem:convex-polar-gauge}(iv) and \cite[Proposition~V.3.3.4]{hiriart1996convex}, we have
\begin{equation}
\label{eq:sigw}
\begin{aligned}
J^\circ\bpa{-\XX^\top\nabla F(\XX\xx_0,\yy)} &= \sigma_{\calC}\bpa{-\XX^\top\nabla F(\XX\xx_0,\yy)} = \sup_{\bs z \in \XX(\calC)} -\sum_{i=1}^n  f_i'((\XX\xx_0)_i,\yy_i) \bs z_i .
\end{aligned}
\end{equation}
Thus, checking the event $\Escr$ amounts to establishing a deviation inequality for the supremum of an empirical process\footnote{As $\XX(\calC)$ is compact, it has a dense countable subset.} above its mean under the weak Bernstein moment condition~\ref{assump:ygenF}(iii), which essentially requires that the $f_i'((\XX\xx_0)_i,\yy_i)$ have sub-exponential tails, We will first tackle the case where $\calC$ is the convex hull of a finite set (i.e. $\calC$ is a polytope).

\subsubsection{Polyhedral penalty}\label{subsubsec:polyhedralFgen}
We here suppose that $J$ is a finite-valued gauge of $\calC = \co{\calV}$, where $\calV$ is finite, i.e. $\calC$ is a polytope with vertices~\cite[Corollary~19.1.1]{rockafellar1996convex}. Our first oracle inequality in probability is the following.
\begin{proposition}\label{prop:SOIprobgenF}
Consider the estimators $\thewa$ and $\thvar$, where $F$ and $J \eqdef \gamma_{\calC}$ satisfy Assumptions~\ref{assump:F}, \ref{assump:Fmom}, \ref{assump:J} and \ref{assump:ygenF}, and $\calC$ is a polytope with vertices $\calV$. Suppose that $\rank(\XX)=n$ and \linebreak$\max_{\bsv \in \calV}\norm{\XX\bsv}{\infty} \leq 1$, and take 
\[
\lambda_n \geq \tau\sigma\sqrt{\frac{2\delta\log(|\calV|)}{n}}\pa{1+\sqrt{2}\kappa/\sigma\sqrt{\frac{\delta\log(|\calV|)}{n}}}, 
\]
for some $\tau > 1$ and $\delta > 1$. Then \eqref{eq:SOIEWA} and \eqref{eq:SOIPEN} hold with probability at least $1-2|\calV|^{1-\delta}$.
\end{proposition}

\begin{proof}
In view of Assumptions~\ref{assump:F} and \ref{assump:ygenF}, one can differentiate under the expectation sign (Leibniz rule) to conclude that $\EE{}{F(\XX\cdot,\yy)}$ is $C^1$ at $\xx_0$ and $\nabla \EE{}{F(\XX\xx_0,\yy)} = \XX^\top\EE{}{\nabla F(\XX\xx_0,\yy)}$. As $\xx_0$ minimizes the population risk, one has $\nabla \EE{}{F(\XX\xx_0,\yy)} = 0$. Using the rank assumption on $\XX$, we deduce that
\[
\EE{}{f_i'((\XX\xx_0)_i,\yy_i)} = 0, \quad \forall 1 \leq i \leq n .
\]
Moreover, \eqref{eq:sigw} specializes to
\[
J^\circ\bpa{-\XX^\top\nabla F(\XX\xx_0,\yy)} = \sup_{\bs z \in \XX(\calV)} -\sum_{i=1}^n  f_i'((\XX\xx_0)_i,\yy_i) \bs z_i.
\]
Let $t=\lambda_n n/\tau$. By the union bound and \eqref{eq:sigw}, we have
\begin{align*}
\PP{J^\circ\bpa{-\XX^\top\nabla F(\XX\xx_0,\yy)} \geq t} 
						    &\leq \PP{\max_{\bs z \in \XX(\calV)}~-\sum_{i=1}^n  f_i'((\XX\xx_0)_i,\yy_i) \bs z_i \geq t} \\
						    &\leq |\calV| \max_{\bs z \in \XX(\calV)}\PP{\abs{\sum_{i=1}^n  f_i'((\XX\xx_0)_i,\yy_i) \bs z_i} \geq t} .
\end{align*}
The random variables $\bpa{f_i'((\XX\xx_0)_i,\yy_i) \bs z_i}_i$ are zero-mean independent, and $\forall i$ and $m \geq 2$
\[
\EE{}{\abs{f_i'((\XX\xx_0)_i,\yy_i) \bs z_i}^m} \leq |\bs z_i|^mm! \kappa^{m-2} \sigma_i^2/2 \leq \max_{\bs v \in \calV}\anorm{\XX \bs v}{\infty}^m m! \kappa^{m-2} \sigma_i^2/2  \leq  m! \kappa^{m-2} \sigma_i^2/2 .
\]
We are then in position to apply the Bernstein inequality to get
\begin{align*}
\PP{J^\circ\bpa{-\XX^\top\nabla F(\XX\xx_0,\yy)} \geq t} &\leq 2|\calV| \exp\pa{-\frac{t^2}{2(\kappa t + n \sigma^2)}} ,
\end{align*}
where $\sigma^2=\max_{1 \leq i \leq n} \sigma_i^2$. Every $t$ such that
\[
t \geq \sqrt{\delta\log(|\calV|)}\pa{\kappa \sqrt{\delta\log(|\calV|)} + \sqrt{\kappa^2\delta\log(|\calV|) + 2 n \sigma^2}}, 
\]
satisfies $t^2 \geq 2\delta\log(|\calV|)(\kappa t + n \sigma^2)$. Applying the trivial inequality $\sqrt{a+b} \leq \sqrt{a}+\sqrt{b}$ to the bound on $t$, we conclude. 
\end{proof}

\begin{remark}
In the monograph \cite[Lemma~14.12]{BuhlmannVandeGeerBook11}, the authors derived an exponential deviation inequality for the supremum of an empirical process with finite $\calV$ and possibly unbounded empirical processes under a Bernstein moment condition similar to ours (in fact ours implies theirs). The very last part of our proof can be obtained by applying their result. We detailed it here for the sake of completeness.
\end{remark}

\paragraph{Lasso}
To lighten the notation, let $I_{\xx}=\supp(\xx)$. From \eqref{eq:exlasso}, it is easy to see that
\[
\normOP{\proj_{\T_{\xx}}}{2}{1} = \sqrt{|I_{\xx}|} \qandq J^\circ(\e{\xx}) = \anorm{\sign(\xx_{I_\xx})}{\infty} \leq 1 ,
\]
where last bound holds as an equality whenever $\xx \neq 0$. Further the $\ell_1$ norm is the gauge of the cross-polytope (i.e. the unit $\ell_1$ ball). Its vertex set $\calV$ is the set of unit-norm one-sparse vectors $(\pm \bs a_i)_{1 \leq i \leq p}$, where we recall $(\bs a_i)_{1 \leq i \leq p}$ the canonical basis. Thus
\[
|\calV| = 2p \qandq \max_{\bsv \in \calV} \anorm{\XX \bsv}{2} = \max_{1 \leq i \leq p} \anorm{\XX_i}{2} .
\]
Inserting this into Proposition~\ref{prop:SOIprobgenF}, we obtain the following corollary.

\begin{corollary}\label{cor:SOIgenlasso}
Consider the estimators $\thewa$ and $\thvar$, where where $J$ is the Lasso penalty and $F$ satisfies Assumptions~\ref{assump:F}, \ref{assump:Fmom} and \ref{assump:ygenF}. Suppose that $\rank(\XX)=n$ and $\max_i\anorm{\XX_i}{\infty}\leq 1$, and take 
\[
\lambda_n \geq \tau\sigma\sqrt{\frac{2\delta\log(2p)}{n}}\pa{1+\sqrt{2}\kappa/\sigma\sqrt{\frac{\delta\log(2p)}{n}}}, 
\]
for some $\tau > 1$ and $\delta > 1$. Then, with probability at least $1-2(2p)^{1-\delta}$, the following holds
\begin{align}\label{eq:SOIEWAlassoFgen}
R_n\bpa{\thewa,\xx_0} \leq & \inf_{\substack{I \subset \ens{1,\ldots,p} \\ \xx:~\supp(\xx)=I}} \pa{R_n\bpa{\xx,\xx_0} + \tfrac{1}{n}\varphi^+\pa{\tfrac{\lambda_n\sqrt{n}\pa{\tau+1}\sqrt{|I|}}{\tau\cfac{\Span\ens{\bs a_i}_{i \in I}}{\tfrac{\tau+1}{\tau-1}}}}} + p\beta ,
\end{align}
and
\begin{align}\label{eq:SOIPENlassoFgen}
R_n\bpa{\thvar,\xx_0} \leq & \inf_{\substack{I \subset \ens{1,\ldots,p} \\ \xx:~\supp(\xx)=I}} \pa{R_n\bpa{\xx,\xx_0} + \tfrac{1}{n}\varphi^+\pa{\tfrac{\lambda_n\sqrt{n}\pa{\tau+1}\sqrt{|I|}}{\tau\cfac{\Span\ens{\bs a_i}_{i \in I}}{\tfrac{\tau+1}{\tau-1}}}}} .
\end{align}
\end{corollary}

For $\thvar$, we recover a similar scaling for $\lambda_n$ and the oracle inequality as in~\cite{VandeGeer08}, though in the latter the oracle inequality is not sharp unlike ours. Note that the above oracle inequality extends readily to the case of analysis/fused Lasso $\norm{\bs D^\top \cdot}{1}$ where $\bs D$ is surjective. We leave the details to the interested reader (see also the analysis group Lasso example in Section~\ref{sec:SOIlin}).

\paragraph{Anti-sparsity}
From Section~\ref{subsec:decomplinf}, recall the saturation support $\Isat_{\xx}$ of $\xx$. From \eqref{eq:exlinf}, we get
\[
\normOP{\proj_{\T_{\xx}}}{2}{\infty} = 1 \qandq J^\circ(\e{\xx}) = \norm{\sign(\xx_{\Isat_{\xx}})}{1}/|\Isat_{\xx}| \leq 1 ,
\]
with equality whenever $\xx \neq 0$. In addition, the $\ell_\infty$ norm is the gauge of the hypercube whose vertex set is $\calV = \ens{\pm 1}^p$. Thus
\[
|\calV| = 2^p .
\]

We have the following oracle inequalities.

\begin{corollary}\label{cor:SOIgenlinf}
Consider the estimators $\thewa$ and $\thvar$, where where $J$ is anti-sparsity penalty \eqref{eq:Jlinf}, and $F$ satisfies Assumptions~\ref{assump:F}, \ref{assump:Fmom} and \ref{assump:ygenF}. Suppose that $\rank(\XX)=n$ and $\max_{i,j}|\XX_{i,j}|\leq 1/p$, and take 
\[
\lambda_n \geq \tau\sigma\sqrt{2\delta\log(2)}\sqrt{\frac{p}{n}}\pa{1+2\kappa/\sigma\sqrt{\delta\log(2)}\sqrt{\frac{p}{n}}}, 
\]
for some $\tau > 1$ and $\delta > 1$. Then, with probability at least $1-2^{-p(\delta-1)+1}$, the following holds
\begin{align}\label{eq:SOIEWAgenlinf}
R_n\bpa{\thewa,\xx_0} \leq & \inf_{\substack{I \subset \ens{1,\ldots,p} \\ \xx:~\Isat_{\xx}=I}} \pa{R_n\bpa{\xx,\xx_0} + \tfrac{1}{n}\varphi^+\pa{\tfrac{\lambda_n\sqrt{n}\pa{\tau+1}}{\tau\cfac{\enscond{\bxx}{\bxx_{I} \in \RR \sign(\xx_{I})}}{\tfrac{\tau+1}{\tau-1}}}}} + p\beta ,
\end{align}
and
\begin{align}\label{eq:SOIPENgenlinf}
R_n\bpa{\thvar,\xx_0} \leq & \inf_{\substack{I \subset \ens{1,\ldots,p} \\ \xx:~\Isat_{\xx}=I}} \pa{R_n\bpa{\xx,\xx_0} + \tfrac{1}{n}\varphi^+\pa{\tfrac{\lambda_n\sqrt{n}\pa{\tau+1}}{\tau\cfac{\enscond{\bxx}{\bxx_{I} \in \RR \sign(\xx_{I})}}{\tfrac{\tau+1}{\tau-1}}}}} .
\end{align}
\end{corollary}

We are not aware of any result of this kind in the literature. The bound imposed on $\XX$ is similar to what is generally assumed in the vector quantization literature~\cite{lyubarskii2010uncertainty,studer12signal}.

\subsubsection{General penalty}
Extending the above reasoning to a general penalty requires a deviation inequality for the supremum of an empirical process in~\eqref{eq:sigw} under the Bernstein moment condition~\ref{assump:ygenF}(iii), but without the need of uniform boundedness. This can be achieved via generic chaining along a tree using entropy with bracketing; see~\cite[Theorem~8]{VandeGeer13}. The resulting deviation bound will thus depend on the entropies with bracketing. These quantities capture the complexity of the set $\XX(\calC)$ but are intricate to compute in general. This subject deserves further investigation that we leave to a future work.

\begin{remark}[{\textbf{Group Lasso}}]
Using the union bound, we have
\begin{align*}
\PP{\max_{i \in \ens{1,\ldots,L}} \norm{\XX_{b_i}^\top \bsxi}{2} \geq \lambda_n n/\tau} 
&\leq \sum_{i=1}^L \PP{\norm{\XX_{b_i}^\top\bsxi}{2} \geq \lambda_n n/\tau} .
\end{align*}
This requires a concentration inequality for quadratic forms of independent random variables satisfying the Bernstein moment assumption above. We are not aware of any such a result. But if our moment condition is strengthened to 
\[
\EE{}{\abs{f_i'((\XX\xx_0)_i,\yy_i)}^{2m}} \leq m! \kappa^{2(m-1)} \sigma_i^2/2, \quad \forall 1 \leq i \leq n, \forall m \geq 1, 
\]
then one can use \cite[Theorem~3]{bellec14}. Indeed, assume the nroamlization $\max_i\normOP{\XX_{b_i}^\top\XX_{b_i}}{2}{2}\leq n$, which entails
\[
\EE{}{\norm{\XX_{b_i}^\top \nabla F(\XX\xx_0,\yy)}{2}} \leq \EE{}{\norm{\XX_{b_i}^\top\nabla F(\XX\xx_0,\yy)}{2}^2}^{1/2} \leq \sigma\sqrt{K n/2} .
\] 
It then follows that taking  
\[
\lambda_n \geq \tau \frac{\sigma \sqrt{K} + 16\kappa\sqrt{\delta\log(L)}}{n} , \quad \delta > 1 ,
\]
the oracle inequalities \eqref{eq:SOIEWAglasso} and \eqref{eq:SOIPENglasso} hold for the group Lasso with probability at least $1-L^{1-\delta}$. A similar result can be proved for the analysis group Lasso just as well with a proper normalization assumption on $\XX$ (see Section~\ref{subsubsec:exglasso}).
\end{remark}
\section{Oracle inequalities for low-complexity linear regression}\label{sec:SOIlin}

In this section, we consider the classical linear regression problem where the $n$ response-covariate pairs $(\yy_i,\XX_i)$ are linked as
\begin{equation} \label{eq:obs}
\yy = \XX \xx_0 + \bsxi,
\end{equation}
where $\bsxi$ is a noise vector. The data loss will be set to $F(\bsu,\yy)=\tfrac{1}{2}\norm{\yy-\bsu}{2}^2$. This in turn entails that $\varphi=\varphi^+=\tfrac{1}{2}\pa{\cdot}^2$ on $\RR_+$ and $R_n\bpa{\xx,\xx_0}=\tfrac{1}{2n}\norm{\XX\xx-\XX\xx_0}{2}^2$.

In this section, we assume that the noise $\bsxi$ is a zero-mean sub-Gaussian vector in $\RR^n$ with parameter $\sigma$. That is, its one-dimensional marginals $\dotp{\bsxi}{\bs z}$ are sub-Gaussian random variables $\forall \bs z \in \RR^n$, i.e. they satisfy
\begin{equation}
\label{eq:subgmarg}
\PP{\abs{\dotp{\bsxi}{\bs z}} \geq t} \leq 2 e^{-t^2/(2\anorm{\bs z}{2}^2\sigma^2)}, \quad \forall \bs z \in \RR^n .
\end{equation}

In this case, the bounds of Section~\ref{subsec:SOIprob} can be improved.

\subsection{General penalty}
As we will shortly show, the event $\Escr$ will depend on the Gaussian width, a summary geometric quantity which, informally speaking, measures the size of the bulk of a set in $\RR^n$.
\begin{definition}
The Gaussian width of a subset $\calS \subset \RR^n$ is defined as
\[
w(\calS) \eqdef \EE{}{\sigma_{\calS}(\bs g)}, \qwhereq \bs g \sim \calN(0,\Id_n) .
\]
\end{definition}

The concept of Gaussian width has appeared in the literature in different contexts. In particular, it has been used to establish sample complexity bounds to ensure exact recovery (noiseless case) and mean-square estimation stability (noisy case) for low-complexity penalized estimators from Gaussian measurements; see e.g.~\cite{rudelson2008sparse,chandrasekaran2012convex,TroppChapter14,VershyninChapter14,vaiterbookchap15}. 

The Gaussian width has deep connections to convex geometry and it enjoys many useful properties. It is well-known that it is positively homogeneous, monotonic w.r.t. inclusion, and invariant under orthogonal transformations. Moreover, $w(\co{\calS}) = w(\calS)$. From Lemma~\ref{lem:suppcompact}(ii)-(iii), $w(\calS)$ is a non-negative finite quantity whenever the set $\calS$ is bounded and contains the origin.


We are now ready to state our oracle inequality in probability with sub-Gaussian noise.
\begin{proposition}\label{prop:SOIprobsubgwgen}
Let the data generated by \eqref{eq:obs} where $\bsxi$ is a zero-mean sub-Gaussian random vector with parameter $\sigma$. Consider the estimators $\thewa$ and $\thvar$, where $F$ and $J \eqdef \gamma_{\calC}$ satisfy Assumptions~\ref{assump:F}-\ref{assump:Fmom} and \ref{assump:J}. Suppose that $\lambda_n \geq \frac{\tau\sigma c_1 \sqrt{2\log(c_2/\delta)}w\pa{\XX(\calC)}}{n}<$, for some $\tau > 1$ and $0 < \delta < \min(c_2,1)$, where $c_1$ and $c_2$ are positive absolute constants. Then with probability at least $1-\delta$, \eqref{eq:SOIEWA} and \eqref{eq:SOIPEN} hold with the remainder term given by \eqref{eq:remquad} with $\nu=1$.
\end{proposition}

The proof requires sophisticated ideas from the theory of generic chaining~\cite{talagrandchaining05}, but we only apply these results. The constants $c_1$ and $c_2$ can be traced back to the proof of these results as detailed in~\cite{talagrandchaining05}. 

\begin{proof}
First, from \eqref{eq:subgmarg}, we have the bound
\[
\PP{\abs{\dotp{\bsxi}{\bs z-\bs z'}} \geq t} \leq 2 e^{-t^2/(2\anorm{\bs z - \bs z'}{2}^2\sigma^2)}, \quad \forall \bs z, \bs z' \in \RR^n ,
\]
i.e. the increment condition~\cite[(0.4)]{talagrandchaining05} is verified. Thus combining \eqref{eq:sigw} with the probability bound in~\cite[page~11]{talagrandchaining05}, the generic chaining theorem~\cite[Theorem~1.2.6]{talagrandchaining05} and the majorizing measure theorem~\cite[Theorem~2.1.1]{talagrandchaining05}, we have
\begin{align*}
\PP{J^\circ(\XX^\top\bsxi) \geq \lambda_n n/\tau} 
	&\leq \PP{\sup_{\bs z \in \XX(\calC)} \dotp{\bsxi}{\bs z} \geq \sigma c_1 \sqrt{2\log(c_2/\delta)}w\pa{\XX(\calC)}} \\
	&\leq c_2\exp\pa{-\frac{\sigma^22\log(c_2/\delta)}{2\sigma^2}} = \delta .
\end{align*}
\end{proof}

If the noise is Gaussian, an enhanced version can be proved by invoking Gaussian concentration of Lipschitz functions~\cite{ledouxbook}.

\begin{proposition}\label{prop:SOIprobgw}
Let the data generated by \eqref{eq:obs} with noise $\bsxi \sim \calN(0,\sigma^2\Id_n)$. Consider the estimators $\thewa$ and $\thvar$, where $F$ and $J \eqdef \gamma_{\calC}$ satisfy Assumptions~\ref{assump:F}-\ref{assump:Fmom} and \ref{assump:J}. Suppose that $\lambda_n \geq \frac{(1+\delta)\tau\sigma w\pa{\XX(\calC)}}{n}$, for some $\tau > 1$ and $\delta > 0$. Then with probability at least $1-\exp\pa{-\frac{\delta^2w\pa{\XX(\calC)}^2}{2\normOP{\XX}{J}{2}^2}}$, \eqref{eq:SOIEWA} and \eqref{eq:SOIPEN} hold with the remainder term given by \eqref{eq:remquad} with $\nu=1$.
\end{proposition}

\begin{proof}
Thanks to sublinearity (see Lemma~\ref{lem:convex-gauge}(i) and Lemma~\ref{lem:convex-polar-gauge}), the function $\bsxi \mapsto J^\circ(\XX^\top\bsxi)$ is Lipschitz continuous with Lipschitz constant $\anormOP{\XX^\top}{2}{J^\circ} = \normOP{\XX}{J}{2}$. From~\eqref{eq:sigw}, we also have
\[
\EE{}{J^\circ\bpa{\XX^\top\bsxi}} = \sigma w\pa{\XX(\calC)} .
\]
Observe that $\XX(\calC)$ is a convex compact set containing the origin. Setting $\epsilon=\lambda_n n/\tau - \sigma w\pa{\XX(\calC)} \geq \delta\sigma w\pa{\XX(\calC)}$, it follows from \eqref{eq:sigw} and the Gaussian concentration of Lipschitz functions~\cite{ledouxbook} that
\begin{align*}
\PP{J^\circ(\XX^\top\bsxi) \geq \lambda_n n/\tau} 
	&\leq\PP{J^\circ(\XX^\top\bsxi) - \EE{}{J^\circ(\XX^\top\bsxi)} \geq \epsilon} \\
	&\leq \PP{J^\circ(\XX^\top\bsxi/\sigma) - w\pa{\XX(\calC)} \geq \delta w\pa{\XX(\calC)}} \\
	&\leq \exp\pa{-\frac{\delta^2 w\pa{\XX(\calC)}^2}{2\normOP{\XX}{J}{2}^2}} .
\end{align*}
\end{proof}

Estimating theoretically the Gaussian width of a set\footnote{Not to mention its image with a linear operator as for $\XX(\calC)$.} is a non-trivial problem that has been extensively studied in the areas of probability in Banach spaces and stochastic processes. There are classical bounds on the Gaussian width (Sudakov's and Dudley's inequalities), but they are difficult to estimate in most cases and neither of these bounds is tight for all sets. When the set is a convex cone (intersected with a sphere), tractable estimates based on polarity arguments were proposed in, e.g., \cite{chandrasekaran2012convex}. 

\subsection{Polyhedral penalty}\label{subsec:polyhreg}
When $\calC$ and is polytope, enhanced oracle inequalities can be obtained by invoking a simple union bound argument.

\begin{proposition}\label{prop:SOIprobsubg}
Let the data generated by \eqref{eq:obs} where $\bsxi$ is a zero-mean sub-Gaussian random vector with parameter $\sigma$. Consider the estimators $\thewa$ and $\thvar$, where $F$ and $J \eqdef \gamma_{\calC}$ satisfy Assumptions~\ref{assump:F}-\ref{assump:Fmom} and \ref{assump:J}, and moreover $\calC$ is a polytope with vertices $\calV$. Suppose that $\lambda_n \geq \frac{\tau\sigma\bpa{\max_{\bsv \in \calV} \anorm{\XX \bsv}{2}}\sqrt{2\delta\log(|\calV|)}}{n}$, for some $\tau > 1$ and $\delta > 1$. Then with probability at least $1-2|\calV|^{1-\delta}$, \eqref{eq:SOIEWA} and \eqref{eq:SOIPEN} hold with the remainder term given by \eqref{eq:remquad} with $\nu=1$.

In particular, if $\max_{\bsv \in \calV} \anorm{\XX \bsv}{2} \leq \sqrt{n}$, then one can take $\lambda_n \geq \tau\sigma\sqrt{\frac{2\delta\log(|\calV|)}{n}}$.
\end{proposition}

\begin{proof}
From~\eqref{eq:sigw} we have
\[
J^\circ\bpa{\XX^\top\bsxi} = \max_{\bsv \in \calC}~\dotp{\XX\bsv}{\bsxi} = \max_{\bsv \in \calV}~\dotp{\XX\bsv}{\bsxi},
\]
where in the last inequality, we used the fact that a convex function attains its maximum on $\calC$ at an extreme point $\calV$. 
Let $\epsilon=\sigma\bpa{\max_{\bsv \in \calV} \anorm{\XX \bsv}{2}}\sqrt{2\delta\log(|\calV|)}$. By the union bound, \eqref{eq:subgmarg} and \eqref{eq:sigw}, we have
\begin{align*}
\PP{J^\circ\bpa{\XX^\top\bsxi} \geq \lambda_n n/\tau} &\leq \PP{\max_{\bsv \in \calV}~\dotp{\XX\bsv}{\bsxi} \geq \epsilon} \\
						    &\leq |\calV| \max_{\bsv \in \calV}\PP{\dotp{\XX\bsv}{\bsxi} \geq \epsilon} \\
						    &\leq |\calV| \max_{\bsv \in \calV}\PP{\abs{\dotp{\XX\bsv}{\bsxi}} \geq \epsilon} \\
						    &\leq 2|\calV| \exp\bpa{- \epsilon^2 / \bpa{2\sigma^2 \max_{\bsv \in \calV}\anorm{\XX\bsv}{2}^2}}\\
						    &\leq 2|\calV|^{1-\delta} .
\end{align*}
\end{proof}


\subsection{Applications}\label{subsec:SOIapplin}

In this section, we exemplify our oracle inequalities for the penalties described in Section~\ref{subsec:decompexamp}.

\subsubsection{Lasso}
Recall the derivations for the Lasso in Section~\ref{subsubsec:polyhedralFgen}. We obtain the following corollary of Proposition~\ref{prop:SOIprobsubg}.

\begin{corollary}\label{cor:SOIlinlasso}
Let the data generated by \eqref{eq:obs} where $\bsxi$ is a zero-mean sub-Gaussian random vector with parameter $\sigma$. Assume that $\XX$ is such that $\max_i\anorm{\XX_i}{2}\leq\sqrt{n}$. Consider the estimators $\thewa$ and $\thvar$, where $J$ is the Lasso penalty \eqref{eq:Jlasso} and $F$ satisfies Assumptions~\ref{assump:F}-\ref{assump:Fmom}. Suppose that $\lambda_n \geq \tau\sigma\sqrt{\frac{2\delta\log(2p)}{n}}$, for some $\tau > 1$ and $\delta > 1$. Then, with probability at least $1-2(2p)^{1-\delta}$, the following holds
\begin{align}\label{eq:SOIEWAlasso}
\tfrac{1}{n}\norm{\XX\thewa-\XX\xx_0}{2}^2 \leq & \inf_{\substack{I \subset \ens{1,\ldots,p} \\ \xx:~\supp(\xx)=I}} \pa{\tfrac{1}{n}\norm{\XX\xx-\XX\xx_0}{2}^2 + \tfrac{\lambda_n^2\pa{\tau+1}^2|I|}{\tau^2\cfac{\Span\ens{\bs a_i}_{i \in I}}{\tfrac{\tau+1}{\tau-1}}^2}} + p\beta ,
\end{align}
and
\begin{align}\label{eq:SOIPENlasso}
\tfrac{1}{n}\norm{\XX\thvar-\XX\xx_0}{2}^2 \leq & \inf_{\substack{I \subset \ens{1,\ldots,p} \\ \xx:~\supp(\xx)=I}} \pa{\tfrac{1}{n}\norm{\XX\xx-\XX\xx_0}{2}^2 + \tfrac{\lambda_n^2\pa{\tau+1}^2|I|}{\tau^2\cfac{\Span\ens{\bs a_i}_{i \in I}}{\tfrac{\tau+1}{\tau-1}}^2}} .
\end{align}
\end{corollary}

The remainder term grows as $\tfrac{|I|\log(p)}{n}$. The oracle inequality \eqref{eq:SOIPENlasso} recovers \cite[Theorem~1]{dalalyan16} in the exactly sparse case, and \eqref{eq:SOIPENlasso} the one in \cite[Theorem~4]{sunzhang} (see also \cite[Theorem~11]{koltchinskii11} and \cite[Theorem~2]{dalalyan17}). It is worth mentioning, however, that \cite[Theorem~1]{dalalyan16} handles the inexactly sparse case while we do not.

\subsubsection{Group Lasso}
Recall the notations in Section~\ref{subsec:decompglasso}, and denote $I_{\xx}=\bsupp(\xx)$ the set indexing active blocks in $\xx$. From \eqref{eq:exglasso}, we have
\[
\normOP{\proj_{\T_{\xx}}}{2}{J} = \sqrt{|I_{\xx}|} \qandq J^\circ(\e{\xx}) = \anorm{\e{\xx}}{\infty,2} \leq 1 ,
\]
where the last bound holds as an equality whenever $\xx \neq 0$.

We have the following oracle inequalities as corollaries of Proposition~\ref{prop:SOIprobsubgwgen} and Proposition~\ref{prop:SOIprobgw}.

\begin{corollary}\label{cor:SOIlinglasso}
Let the data generated by \eqref{eq:obs}. Consider the estimators $\thewa$ and $\thvar$, where $F$ satisfies Assumptions~\ref{assump:F}-\ref{assump:Fmom}, and $J$ is the group Lasso \eqref{eq:Jglasso} with $L$ non-overlapping blocks of equal size $K$. Assume that $\XX$ is such that $\max_i\normOP{\XX_{b_i}^\top\XX_{b_i}}{2}{2}\leq n$. 

\begin{enumerate}[label=(\roman*)]
\item $\bsxi$ is a zero-mean sub-Gaussian random vector with parameter $\sigma$: suppose that \linebreak$\lambda_n \geq 3\tau\sigma c_1 \frac{\sqrt{2\log(c_2/\delta)}\pa{\sqrt{K}+\sqrt{2\log(L)}}}{\sqrt{n}}$, for some $\tau > 1$ and $0 < \delta < \min(c_2,1)$, where $c_1$ and $c_2$ are the positive absolute constants in Proposition~\ref{prop:SOIprobsubgwgen}. Then, with probability at least $1-\delta$, the following holds
\begin{align}\label{eq:SOIEWAglasso}
\tfrac{1}{n}\norm{\XX\thewa-\XX\xx_0}{2}^2 \leq & \inf_{\substack{I \subset \ens{1,\ldots,L} \\ \xx:~\bsupp(\xx)=I}} \pa{\tfrac{1}{n}\norm{\XX\xx-\XX\xx_0}{2}^2 + \tfrac{\lambda_n^2\pa{\tau+1}^2|I|}{\tau^2\cfac{\Span\ens{a_j}_{j \in b_i, i \in I}}{\tfrac{\tau+1}{\tau-1}}^2}} + p\beta ,
\end{align}
and
\begin{align}\label{eq:SOIPENglasso}
\tfrac{1}{n}\norm{\XX\thvar-\XX\xx_0}{2}^2 \leq & \inf_{\substack{I \subset \ens{1,\ldots,L} \\ \xx:~\bsupp(\xx)=I}} \pa{\tfrac{1}{n}\norm{\XX\xx-\XX\xx_0}{2}^2 + \tfrac{\lambda_n^2\pa{\tau+1}^2|I|}{\tau^2\cfac{\Span\ens{a_j}_{j \in b_i, i \in I}}{\tfrac{\tau+1}{\tau-1}}^2}} .
\end{align}

\item $\bsxi \sim \calN(0,\sigma^2\Id_n)$: suppose that $\lambda_n \geq \tau\sigma\frac{\sqrt{K}+\sqrt{2\delta\log(L)}}{\sqrt{n}}$, for some $\tau > 1$ and $\delta > 1$. Then, with probability at least $1-L^{1-\delta}$, \eqref{eq:SOIEWAglasso} and \eqref{eq:SOIPENglasso} hold.
\end{enumerate}
\end{corollary}

The first remainder term is on the order $\frac{|I|\pa{\sqrt{K}+\sqrt{2\log(L)}}^2}{n}$. This is similar to the scaling that has been provided in the literature for EWA with other group sparsity priors and noises~\cite{rigo12,luu16}. Similar rates were given for $\thvar$ with the group Lasso in \cite{negahban2010unified,lounici11,vandegeer14}. 

\begin{proof}
{$~$ \vspace*{1pt}}
\begin{enumerate}[label=(\roman*)]
\item This is a consequence of Proposition~\ref{prop:SOIprobsubgwgen}, for which we need to bound 
\[
w(\XX(\calC)) = \EE{}{\max_{i \in \ens{1,\ldots,L}} \norm{\XX_{b_i}^\top \bs g}{2}} .
\]
We first have, for any block $b_i$
\[
\EE{}{\norm{\XX_{b_i}^\top\bs g}{2}} \leq \EE{}{\norm{\XX_{b_i}^\top\bs g}{2}^2}^{1/2} \leq \sqrt{K n} .
\] 
Furthermore,  $\norm{\XX_{b_i}^\top \cdot}{2}$ is Lipschitz continuous with Lipschitz constant $\normOP{\XX_{b_i}}{2}{2} \leq \sqrt{n}$. 
Thus the union bound and Gaussian concentration of Lipschitz functions~\cite{ledouxbook} yield, for any $t > 0$,
\begin{align*}
\PP{\max_{i \in \ens{1,\ldots,L}} \norm{\XX_{b_i}^\top \bs g}{2} \geq \sqrt{K n} + t} 
&\leq \sum_{i=1}^L \PP{\norm{\XX_{b_i}^\top\bsxi}{2} - \EE{}{\norm{\XX_{b_i}^\top\bsxi}{2}} \geq t}
\leq L \exp\pa{-\frac{t^2}{2n}} .
\end{align*}
Let $\kappa = \sqrt{Kn}+\sqrt{2n\log(L)}$. $w(\XX(\calC))$ can be expressed as
\begin{align*}
w(\XX(\calC)) &= \int_{0}^{\infty} \PP{\max_{i \in \ens{1,\ldots,L}} \norm{\XX_{b_i}^\top \bs g}{2} \geq s} ds \leq \int_{0}^{\kappa} ds + \int_{\kappa}^\infty e^{-\frac{(s-\sqrt{Kn})^2-2n\log(L)}{2n}} ds \\
	      &= \kappa + \sqrt{n}\int_{\kappa/\sqrt{n}}^\infty e^{-\frac{(s-\sqrt{K})^2-2\log(L)}{2}} ds \\
	      &\leq \kappa + \sqrt{n}\int_{\kappa/\sqrt{n}}^\infty e^{-\frac{s-\kappa/\sqrt{n}}{2}} ds
	      = \kappa + 2\sqrt{n} \leq 3\kappa .
\end{align*} 

\item The proof follows the lines of Proposition~\ref{prop:SOIprobgw} where we additionally use the union bound. Indeed, 
\begin{align*}
\PP{\max_{i \in \ens{1,\ldots,L}} \norm{\XX_{b_i}^\top \bsxi}{2} \geq \lambda_n n/\tau} 
&\leq \sum_{i=1}^L \PP{\norm{\XX_{b_i}^\top\bsxi}{2} - \EE{}{\norm{\XX_{b_i}^\top\bsxi}{2}} \geq \lambda_n n/\tau - \EE{}{\norm{\XX_{b_i}^\top\bsxi}{2}}} \\
&\leq \sum_{i=1}^L \PP{\norm{\XX_{b_i}^\top\bsxi}{2} - \EE{}{\norm{\XX_{b_i}^\top\bsxi}{2}} \geq \lambda_n n/\tau - \sigma\sqrt{Kn}} \\
&\leq \sum_{i=1}^L \PP{\norm{\XX_{b_i}^\top\bsxi}{2} - \EE{}{\norm{\XX_{b_i}^\top\bsxi}{2}} \geq \sigma\sqrt{2\delta n\log(L)}} \\
&\leq L \exp\pa{-\delta\log(L)} = L^{1-\delta} ,
\end{align*}
where used the Gaussian concentration of Lipschitz functions~\cite{ledouxbook} in the last inequality.
\end{enumerate}
\end{proof}

We observe in passing that another way to prove the oracle inequalities in the sub-Gaussian is to use Dudley's inequality on the sphere in $\RR^K$ after applying a union bound on the $L$ blocks. In addition, in the Gaussian case, the (similar) bound $\lambda_n \geq 3\delta\tau\sigma\frac{\sqrt{K}+\sqrt{2\log(L)}}{\sqrt{n}}$ can be obtained by combining Proposition~\ref{prop:SOIprobgw} and the estimate $w(\XX(\calC)) \leq 3(\sqrt{Kn}+\sqrt{2n\log(L)})$ in the proof of (i). The corresponding probability of success would be at least $1-L^{-9(\delta-1)^2}$.

\subsubsection{Analysis group Lasso}\label{subsubsec:exglasso}
We now turn to the prior penalty \eqref{eq:Janaglasso}. Recall the notations in Section~\ref{subsec:decompanaglasso}, and remind $\Lambda_{\xx}=\bigcup_{i \in \bsupp(\bsD^\top\xx)} b_i$. We assume that $\bsD$ is a frame of $\RR^p$, hence surjective, meaning that there exist $c, d > 0$ such that for any $\bsom \in \RR^p$
\[
d \anorm{\bsom}{2}^2 \leq \norm{\bsD^\top \bsom}{2}^2 \leq c \anorm{\bsom}{2}^2 .
\]
This together with \eqref{eq:exanaglasso}-\eqref{eq:exanaglassoJo} and Cauchy-Schwarz inequality entail
\begin{align*}
\normOP{\proj_{\T_{\xx}}}{2}{J} = \sup_{\anorm{\bsom_{\T_{\xx}}}{2} \leq 1} \norm{\bsD^\top \bsom_{\T_{\xx}}}{1,2}
				&\leq \sqrt{c}\sup_{\anorm{\bsD^\top\bsom_{\T_{\xx}}}{2} \leq 1} \norm{\bsD^\top \bsom_{\T_{\xx}}}{1,2} \\
				&= \sqrt{c}\sup_{\anorm{\bsD_{\Lambda_{\xx}}^\top\bsom_{\T_{\xx}}}{2} \leq 1} \norm{\bsD_{\Lambda_{\xx}}^\top \bsom_{\T_{\xx}}}{1,2} \\
				&= \sqrt{c}\sqrt{|\bsupp(\bsD^\top\xx)|} .
\end{align*}
Note, however, that from \eqref{eq:exanaglasso}, we do not have in general $\anorm{\bsD^+\proj_{\Ker(\bsD^\top_{\Lambda^c_{\xx}})}\bsD\e{\bsD^\top \xx}^{\anorm{}{1,2}}}{\infty,2} \leq 1$. 

With exactly the same arguments to those for proving Corollary~\ref{cor:SOIlinglasso}, replacing $\XX$ by $\XX\bsD$, we arrive at the following oracle inequalities.

\begin{corollary}\label{cor:SOIlinanaglasso}
Let the data generated by \eqref{eq:obs}. Consider the estimators $\thewa$ and $\thvar$, where $F$ satisfies Assumptions~\ref{assump:F}-\ref{assump:Fmom}, and $J$ is the analysis group Lasso \eqref{eq:Janaglasso} with $L$ blocks of equal size $K$. Assume that $\bsD$ is a frame, and $\XX$ is such that $\max_i\normOP{\bsD_{b_i}^\top\XX^\top\XX\bsD_{b_i}}{2}{2}\leq n$.
\begin{enumerate}[label=(\roman*)]
\item $\bsxi$ is a zero-mean sub-Gaussian random vector with parameter $\sigma$: suppose that \linebreak$\lambda_n \geq 3\tau\sigma c_1 \frac{\sqrt{\log(c_2/\delta)}\pa{\sqrt{K}+\sqrt{2\log(L)}}}{\sqrt{n}}$, for some $\tau > 1$ and $0 < \delta < \min(c_2,1)$, where $c_1$ and $c_2$ are the positive absolute constants in Proposition~\ref{prop:SOIprobsubgwgen}. Then, with probability at least $1-\delta$, the following holds
\begin{align}\label{eq:SOIEWAanaglasso}
\tfrac{1}{n}\norm{\XX\thewa-\XX\xx_0}{2}^2 \leq & \inf_{\substack{I \subset \ens{1,\ldots,L} \\ \xx:~\bsupp(\bsD^\top\xx)=I}} \Vast(\tfrac{1}{n}\norm{\XX\xx-\XX\xx_0}{2}^2 \nonumber\\
& + \tfrac{c\lambda_n^2\bpa{\tau\norm{\bsD^+\proj_{\Ker(\bsD^\top_{\Lambda^c_{\xx}})}\bsD\e{\bsD^\top \xx}^{\anorm{}{1,2}}}{\infty,2}+1}^2|I|}{\tau^2\acfac{\Ker(\bsD^\top_{\Lambda^c_{\xx}})}{\tfrac{\tau\norm{\bsD^+\proj_{\Ker(\bsD^\top_{\Lambda^c_{\xx}})}\bsD\e{\bsD^\top \xx}^{\anorm{}{1,2}}}{\infty,2}+1}{\tau-1}}^2}\Vast) + p\beta , 
\end{align}
and
\begin{align}\label{eq:SOIPENanaglasso}
\tfrac{1}{n}\norm{\XX\thewa-\XX\xx_0}{2}^2 \leq & \inf_{\substack{I \subset \ens{1,\ldots,L} \\ \xx:~\bsupp(\bsD^\top\xx)=I}} \Vast(\tfrac{1}{n}\norm{\XX\xx-\XX\xx_0}{2}^2 \nonumber\\
& + \tfrac{c\lambda_n^2\bpa{\tau\norm{\bsD^+\proj_{\Ker(\bsD^\top_{\Lambda^c_{\xx}})}\bsD\e{\bsD^\top \xx}^{\anorm{}{1,2}}}{\infty,2}+1}^2|I|}{\tau^2\acfac{\Ker(\bsD^\top_{\Lambda^c_{\xx}})}{\tfrac{\tau\norm{\bsD^+\proj_{\Ker(\bsD^\top_{\Lambda^c_{\xx}})}\bsD\e{\bsD^\top \xx}^{\anorm{}{1,2}}}{\infty,2}+1}{\tau-1}}^2}\Vast)
\end{align}

\item $\bsxi \sim \calN(0,\sigma^2\Id_n)$: suppose that $\lambda_n \geq \tau\sigma\frac{\sqrt{K}+\sqrt{2\delta\log(L)}}{\sqrt{n}}$, for some $\tau > 1$ and $\delta > 1$. Then, with probability at least $1-L^{1-\delta}$, \eqref{eq:SOIEWAanaglasso} and \eqref{eq:SOIPENanaglasso} hold.
\end{enumerate}
\end{corollary}

To the best of our knowledge, this result is new to the literature. The scaling of the remainder term is the same as in \cite{luu16} and \cite{rigo12} with analysis sparsity priors different from ours (the authors in the latter also assume that $\bsD$ is invertible).

\subsubsection{Anti-sparsity}
Recall the derivations for the $\ell_\infty$ norm example in Section~\ref{subsubsec:polyhedralFgen}. We have the following oracle inequalities from Proposition~\ref{prop:SOIprobsubg}.

\begin{corollary}\label{cor:SOIlinlinf}
Let the data generated by \eqref{eq:obs} where $\bsxi$ is a zero-mean sub-Gaussian random vector with parameter $\sigma$. Assume that $\XX$ is such that $\max_{i,j}|\XX_{i,j}|\leq 1/p$. Consider the estimators $\thewa$ and $\thvar$, where $F$ satisfies Assumptions~\ref{assump:F}-\ref{assump:Fmom}, and $J$ is the anti-sparsity penalty \eqref{eq:Jlinf}. Suppose that $\lambda_n \geq \tau\sigma\sqrt{2\delta\log(2)}\sqrt{\frac{p}{n}}$, for some $\tau > 1$ and $\delta > 1$. Then, with probability at least $1-2^{-p(\delta-1)+1}$, the following holds
\begin{align}\label{eq:SOIEWAlinf}
\tfrac{1}{n}\norm{\XX\thewa-\XX\xx_0}{2}^2 \leq & \inf_{\substack{I \subset \ens{1,\ldots,p} \\ \xx:~\Isat_{\xx}=I}} \pa{\tfrac{1}{n}\norm{\XX\xx-\XX\xx_0}{2}^2 + \tfrac{\lambda_n^2\pa{\tau+1}^2}{\tau^2\cfac{\enscond{\bxx}{\bxx_{I} \in \RR \sign(\xx_{I})}}{\tfrac{\tau+1}{\tau-1}}^2}} + p\beta ,
\end{align}
and
\begin{align}\label{eq:SOIPENlinf}
\tfrac{1}{n}\norm{\XX\thvar-\XX\xx_0}{2}^2 \leq & \inf_{\substack{I \subset \ens{1,\ldots,p} \\ \xx:~\Isat_{\xx}=I}} \pa{\tfrac{1}{n}\norm{\XX\xx-\XX\xx_0}{2}^2 + \tfrac{\lambda_n^2\pa{\tau+1}^2}{\tau^2\cfac{\enscond{\bxx}{\bxx_{I} \in \RR \sign(\xx_{I})}}{\tfrac{\tau+1}{\tau-1}}^2}} .
\end{align}
\end{corollary}

The first remainder term scales as $\tfrac{p}{n}$ which reflects that anti-sparsity regularization requires an overdetermined regime to ensure good stability performance. This is in agreement with~\cite[Theorem~7]{vaiterimaiai13}. This phenomenon was also observed by~\cite{donoho2010counting} who studied sample complexity thresholds for noiseless recovery from random projections of the hypercube.


\subsubsection{Nuclear norm}
We now turn to the nuclear norm case. Recall the notations of Section~\ref{subsec:decompnuc}. For matrices $\xx \in \RR^{p_1 \times p_2}$, a measurement map $\XX$ takes the form of a linear operator whose $i$th component is given by the Frobenius scalar product
\[
\XX(\xx)_i = \tr((\XX^i)^\top \xx) = \dotp{\XX^i}{\xx}_{\mathrm{F}} ,
\]
where $\XX^i$ is a matrix in $\RR^{p_1 \times p_2}$. We denote $\anorm{\cdot}{\mathrm{F}}$ the associated norm. From \eqref{eq:exnuc}, it is immediate to see that whenever $\xx \neq 0$,
\[
J^\circ(\e{\xx}) = \anormOP{\bs U \bs V^\top}{2}{2} = 1.
\]
Moreover, from \eqref{eq:exnuc}, we have
\[
\normOP{\proj_{\T_{\xx}}}{\mathrm{F}}{*} = \sup_{\xx' \in \T_{\xx}} \frac{\anorm{\xx'}{*}}{\anorm{\xx'}{\mathrm{F}}} = \sup_{\xx' \in \T_{\xx}} \frac{\anorm{\uplambda(\xx')}{1}}{\anorm{\uplambda(\xx')}{2}} \leq \sup_{\xx' \in \T_{\xx}} \sqrt{\rank(\xx')} \leq \sqrt{\min(r,p_1) + \min(r,p_2)} \leq \sqrt{2r}.
\]
To apply Proposition~\ref{prop:SOIprobsubgwgen} and Proposition~\ref{prop:SOIprobgw}, we need to bound  $w(\XX(\calC))$ ($\calC$ is the nuclear ball), or equivalently, to bound
\[
\EE{}{\normOP{\XX^*(\bs g)}{2}{2}}=\EE{}{\normOP{\sum_{i=1}^n\XX^i\bs g_i}{2}{2}} ,  \quad \bs g \sim \calN(0,\sigma^2\Id_n) ,
\] 
which is the expectation of the operator norm of a random series with matrix coefficients. Thus using \cite[Theorem~4.1.1(4.1.5)]{tropp15} to get this bound, and inserting it into Proposition~\ref{prop:SOIprobsubgwgen} and Proposition~\ref{prop:SOIprobgw}, we get the following oracle inequalities for the nuclear norm. Define 
\[
v(\XX) = \max\pa{\normOP{\sum_{i=1}^n \XX^i(\XX^i)^\top}{2}{2},\normOP{\sum_{i=1}^n (\XX^i)^\top\XX^i}{2}{2}} .
\]

\begin{corollary}\label{cor:SOIlinnuc}
Let the data generated by \eqref{eq:obs} with a linear operator $\XX: \RR^{p_1 \times p_2} \to \RR^n$. Assume that $v(\XX) \leq n$. Consider the estimators $\thewa$ and $\thvar$, where $F$ satisfies Assumptions~\ref{assump:F}-\ref{assump:Fmom}, and $J$ is the nuclear norm \eqref{eqjJnuc}. 
\begin{enumerate}[label=(\roman*)]
\item $\bsxi$ is a zero-mean sub-Gaussian random vector with parameter $\sigma$: suppose that \linebreak$\lambda_n \geq 2\tau\sigma c_1 \sqrt{\frac{\log(c_2/\delta)\log(p_1+p_2)}{n}}$, for some $\tau > 1$ and $0 < \delta < \min(c_2,1)$, where $c_1$ and $c_2$ are the positive absolute constants in Proposition~\ref{prop:SOIprobsubgwgen}. Then, with probability at least $1-\delta$, the following holds
\begin{align}\label{eq:SOIEWAnuc}
\tfrac{1}{n}\norm{\XX\thewa-\XX\xx_0}{2}^2 \leq & \inf_{\substack{r \in \ens{1,\ldots,\min(p_1,p_2)} \\ \xx:~\rank(\xx)=r}} \pa{\tfrac{1}{n}\norm{\XX\xx-\XX\xx_0}{2}^2 + \tfrac{2\lambda_n^2\pa{\tau+1}^2 r}{\tau^2\cfac{\T_{\xx}}{\tfrac{\tau+1}{\tau-1}}^2}} + p_1p_2\beta ,
\end{align}
and
\begin{align}\label{eq:SOIPENnuc}
\tfrac{1}{n}\norm{\XX\thvar-\XX\xx_0}{2}^2 \leq & \inf_{\substack{r \in \ens{1,\ldots,\min(p_1,p_2)} \\ \xx:~\rank(\xx)=r}} \pa{\tfrac{1}{n}\norm{\XX\xx-\XX\xx_0}{2}^2 + \tfrac{2\lambda_n^2\pa{\tau+1}^2 r}{\tau^2\cfac{\T_{\xx}}{\tfrac{\tau+1}{\tau-1}}^2}} .
\end{align}

\item $\bsxi \sim \calN(0,\sigma^2\Id_n)$: suppose that $\lambda_n \geq (1+\delta)\tau\sigma\sqrt{\frac{2\log(p_1+p_2)}{n}}$, for some $\tau > 1$ and $\delta > 0$. Then, with probability at least $1-(p_1+p_2)^{-\delta^2}$, \eqref{eq:SOIEWAnuc} and \eqref{eq:SOIPENnuc} hold.
\end{enumerate}
\end{corollary} 

The set over which the infimum is taken just reminds us that the nuclear norm is partly smooth (see above) relative to the constant rank manifold (which is a Riemannian submanifold of $\RR^{p_1 \times p_2}$)~\cite[Theorem~3.19]{daniilidis2013orthogonal}. The first remainder term now scales as $\frac{r\log(p_1+p_2)}{n}$. In the iid Gaussian case, we recover the same rate as in \cite[Theorem~3]{dalalyan16} for $\thewa$ and in \cite[Theorem~2]{koltchinskii11} for $\thvar$.

\subsection{Discussion of minimax optimality}\label{subsec:minimaxopt}
In this section, we discuss the optimality of the estimators $\thewa$ and $\thvar$ (we remind the reader that the design $\XX$ is fixed). Recall the discussion on stratification at the end of Section~\ref{subsec:SOIEWA}. Let $\calM_0 \in \Mscr$ be the stratum active at $\xx_0 \in \calM_0$. In this setting, with $\beta=O(1/(pn))$, \eqref{eq:SOIEWAps} and Proposition~\ref{prop:SOIprobgw} ensure that
\begin{align*}
\tfrac{1}{n}\norm{\XX\thewa-\XX\xx_0}{2}^2 &\leq  \frac{(1+\delta)^2\sigma^2w\pa{\XX(\calC)}^2}{n^2}\pa{\sup_{\xx \in \calM_0} \frac{\anormOP{\proj_{\T_{\xx}}}{2}{J}^2\bpa{\tau J^\circ(\e{\xx})+1}^2}{\cfac{\T_{\xx}}{\tfrac{\tau J^\circ(\e{\xx})+1}{\tau-1}}^2}} + O\pa{\frac{1}{n}} \\
\tfrac{1}{n}\norm{\XX\thvar-\XX\xx_0}{2}^2 &\leq  \frac{(1+\delta)^2\sigma^2w\pa{\XX(\calC)}^2}{n^2}\pa{\sup_{\xx \in \calM_0} \frac{\anormOP{\proj_{\T_{\xx}}}{2}{J}^2\bpa{\tau J^\circ(\e{\xx})+1}^2}{\cfac{\T_{\xx}}{\tfrac{\tau J^\circ(\e{\xx})+1}{\tau-1}}^2}} ,
\end{align*}
with high probability. In particular, for a polyhedral gauge penalty, in which case $\calM_0=\T_{\xx_0}$ (see~\cite{vaiterimaiai13}), and under the normalization $\max_{\bsv \calV} \anorm{\XX\bsv}{2} \leq \sqrt{n}$, Proposition~\ref{prop:SOIprobsubg} entails
\begin{align*}
\tfrac{1}{n}\norm{\XX\thewa-\XX\xx_0}{2}^2 &\leq C\frac{2\delta\sigma^2\anormOP{\proj_{\calM_0}}{2}{J}^2\log(|\calV|)}{n}\pa{\sup_{\xx \in \calM_0}\frac{\bpa{\tau J^\circ(\e{\xx})+1}^2}{\cfac{\calM_0}{\tfrac{\tau J^\circ(\e{\xx})+1}{\tau-1}}^2}} \\
\tfrac{1}{n}\norm{\XX\thvar-\XX\xx_0}{2}^2 &\leq  \frac{2\delta\sigma^2\anormOP{\proj_{\calM_0}}{2}{J}^2\log(|\calV|)}{n}\pa{\sup_{\xx \in \calM_0}\frac{\bpa{\tau J^\circ(\e{\xx})+1}^2}{\cfac{\calM_0}{\tfrac{\tau J^\circ(\e{\xx})+1}{\tau-1}}^2}} ,
\end{align*}
with high probability. Thus the risk bounds only depend on $\calM_0$. A natural question that arises is whether the above bounds are optimal, i.e. whether an estimator can achieve a significantly better prediction risk than $\thewa$ and $\thvar$ uniformly on $\calM_0$. A classical way to answer this question is the minimax point of view. This amounts to finding a lower bound on the minimax probabilities of the form
\[
\inf_{\hxx} \sup_{\xx \in \calM_0} \Pr\pa{\tfrac{1}{n}\norm{\XX\hxx-\XX\xx}{2}^2 \geq \psi_n} ,
\] 
where $\psi_n$ is the rate, which ideally, should be comparable to the risk bounds above. A standard path to derive such a lower bound is to exhibit a subset of $\calM_0$ of well-separated points while controlling its diameter, see \cite[Chapter~2]{tsybakovbook} or \cite[Section~4.3]{massartbook}. This however must be worked out on a case-by-case basis.

\begin{example}\label{ex:minimaxlasso}
For the Lasso case, $\calM_0 = \T_{\xx_0}$ is the subspace of vectors whose support is contained in that of $\xx_0$. Let $I=\supp(\xx_0)$ and $s = \anorm{\xx_0}{0}$. Define the set 
\[
\calB_0 = \enscond{\xx \in \RR^p}{\xx_{I} \in \ens{0,1}^{s} \qandq \xx_{I^c}=0} .
\]
We have $\calB_0 \subset \calM_0$ and $\anorm{\xx-\xx'}{0} \leq 2s$ for all $(\xx,\xx') \in \calB_0$. Define $\calF_0 \eqdef \enscond{r\XX\xx}{\xx \in \calB_0}$, for $r > 0$ to be specified later. Due to the Varshamov-Gilbert lemma~\cite[Lemma~4.7]{massartbook}, given $a \in ]0,1[$, there exists a subset $\calB \subset \calB_0$ with cardinality $|\calB| \geq 2^{\rho s/2}$ such that for two distinct elements $\XX\xx$ and $\XX\xx'$ in $\calF_0$
\begin{align*}
\anorm{\XX(\xx-\xx')}{2}^2 &\geq \infkap r^2 \anorm{\xx-\xx'}{2}^2 \geq 2(1-a)\infkap r^2 s , \\
\anorm{\XX(\xx-\xx')}{2}^2 &\leq \supkap r^2 \anorm{\xx-\xx'}{2}^2 \leq 4 \supkap r^2 s) ,
\end{align*}
where
\[
\infkap = \inf_{\xx \in \calM_0} \frac{\anorm{\XX\xx}{2}^2}{\anorm{\xx}{2}^2} \leq \supkap = \sup_{\xx \in \calM_0} \frac{\anorm{\XX\xx}{2}^2}{\anorm{\xx}{2}^2} .
\]
Standard results from random matrix theory ensure that $\infkap > 0$ for a Gaussian design with high probability as long as $n \geq s + C\sqrt{s}$~\cite{TroppChapter14} for some positive absolute constant $C$.

Then choosing $r^2=\frac{c\rho\sigma^2}{4\supkap}$, where $c \in ]0,1/8[$ and $\rho=(1+a)\log(1+a)+(1-a)\log(1-a)$, we get the bounds
\begin{align*}
\anorm{\XX(\xx-\xx')}{2}^2 &\geq \frac{\sigma^2c(1-a)\rho\infkap}{2\supkap} s , \\
\anorm{\XX(\xx-\xx')}{2}^2 &\leq 2\sigma^2c\log(|\calB|) .
\end{align*}
We are now in position to apply \cite[Theorem~2.5]{tsybakovbook} to conclude that there exists $\eta \in ]0,1[$ (that depends on $a$) such that 
\[
\inf_{\hxx} \sup_{\xx \in \calM_0} \Pr\pa{\tfrac{1}{n}\norm{\XX\hxx-\XX\xx}{2}^2 \geq \frac{\sigma^2c(1-a)\rho\infkap}{4\supkap}\frac{s}{n}} \geq \eta .
\]
This lower bound together with Corollary~\ref{cor:SOIlinlasso} show that $\thewa$ (with $\beta=O(1/(pn))$) and $\thvar$ are nearly minimax (up to a logarithmic factor) over $\calM_0$. 

One can generalize this reasoning to get a minimax lower bound over the larger class of $s$-sparse vectors, i.e. $\bigcup\enscond{V=\Span \ens{(\bs a_j)_{1 \leq j \leq p}}}{\dim(V)=s}$, which is a finite union of subspaces that contains $\calM_0$. Let $(a,b) \in ]0,1[^2$ such that $1 \leq s \leq abp$ and $a(-1+b-\log(b)) \geq \log(2)$ \footnote{E.g. take $b=1/(1+e\sqrt[a]{2})$.}, $c \in ]0,1/8[$. Then combining \cite[Theorem~2.5]{tsybakovbook} and \cite[Lemma~4.6 and Lemma~4.10]{massartbook}, we have for $\eta \eqdef \frac{1}{1+(ab)^{\rho s/2}}\pa{1-2c-\sqrt{\frac{2c}{-\rho\log(ab)}}} \in ]0,1[$
\[
\inf_{\hxx} \sup_{\xx \in \calM_0} \Pr\pa{\tfrac{1}{n}\norm{\XX\hxx-\XX\xx}{2}^2 \geq \frac{\sigma^2c\rho(1-\alpha)\infkap}{2\supkap}\frac{s\log(p/s)}{n}} \geq \eta,
\]
where $\rho=-a(-1+b-\log(b))/\log(ab)$, and $\infkap$ and $\supkap$ are now the restricted isometry constants of $\XX$ of degree $2s$, i.e.
\[
\infkap = \inf_{\anorm{\xx}{0} \leq 2s} \frac{\anorm{\XX\xx}{2}^2}{\anorm{\xx}{2}^2} \leq \supkap = \sup_{\anorm{\xx}{0} \leq 2s} \frac{\anorm{\XX\xx}{2}^2}{\anorm{\xx}{2}^2} .
\]
For this lower bound to be meaningful, $\infkap$ should be positive. From the compressed sensing literature, many random designs are known to verify this condition for $n$ large enough compared to $s$, e.g. sub-Gaussian designs with $n \gtrsim s\log(p)$. 

One can see that the difference between this lower bound and the one on $\calM_0$ lies in the $\log(p/s)$ factor, which basically derives from the control over the union of subspaces. The minimax prediction risk (in expectation) over the $\ell_0$-ball were studied in~\cite{rigo11,raskutti11,verzelen12,yezhang10,wang14}, where similar lower bounds were obtained. 
\end{example}

\begin{example}\label{ex:minimaxglasso}
For the group Lasso with $L$ groups of equal size $K$, $\calM_0$ is the subspace group sparse vectors whose group support is included in that of $\xx_0$. Let $s$ be the number of non-zero (active) groups in $\xx_0$. Following exactly the same reasoning as for the Lasso, one can show that the risk lower bound in probability scales as $C\sigma^2sK/n$, which together with Corollary~\ref{cor:SOIlinglasso}, shows that $\thewa$ and $\thvar$ are nearly minimax (up again to a logarithmic factor) over $\calM_0$. One can also derive the lower bound $C\sigma^2s(K+\log(L/s))/n$ over the set of $s$-block sparse vectors. Such minimax lower bound is comparable to the one in~\cite{lounici11}.
\end{example}

\begin{example}\label{ex:minimaxlinf}
Let's consider the $\ell_\infty$-penalty. Denote the saturation support of $\xx_0$ as $\Isat$ and recall the subspace $\T_{\xx_0}$ form \eqref{eq:exlinf}. Thus, $\calM_0=\T_{\xx_0}$ is the subspace of vectors which are collinear to $\sign(\xx_0)$ on $\Isat$ and free on its complement. Observe that $\dim(\calM_0) = p-s+1$, where $s=|\Isat|$. Define the set
\[
\calB_0 = \enscond{\xx \in \RR^p}{\xx_{\Isat} = \sign(\xx_{\Isat}) \qandq \xx_{(\Isat)^c} \in \ens{0,1}^{p-s})} .
\]
By construction, $\calB_0 \subset \calM_0$, and $\anorm{\xx-\xx'}{0} \leq 2(p-s)$ for all $(\xx,\xx') \in \calB_0$.  Thus following the same arguments as for the Lasso example (using again Varshamov-Gilbert lemma and \cite[Theorem~2.5]{tsybakovbook}), we conclude that there exists $\eta \in ]0,1[$ (that depends on $a$) such that 
\[
\inf_{\hxx} \sup_{\xx \in \calM_0} \Pr\pa{\tfrac{1}{n}\norm{\XX\hxx-\XX\xx}{2}^2 \geq \frac{\sigma^2c(1-a)\rho\infkap}{4\supkap}\frac{p-s}{n}} \geq \eta ,
\]
where the restricted isometry constants are defined similarly to the Lasso but with respect to the model subspace $\calM_0$ of the $\ell_\infty$ norm. Again, for a Gaussian design, $\infkap > 0$ with high probability as long as $n \geq (p-s+1) + C\sqrt{p-s+1}$~\cite{TroppChapter14}.

The obtained minimax lower bound is consistent with the sample complexity thresholds derived in~\cite{donoho2010counting} for noiseless recovery from random projections of the hypercube. For a saturation support size small compared to $p$, the bound of Corollary~\ref{cor:SOIlinlinf} comes close to the minimax lower bound.

\end{example}

\begin{example}\label{ex:minimaxnuc}
Let $r=\rank(\xx_0)$, where $\xx_0 \in \RR^{p_1 \times p_2}$, and $p=\max(p_1,p_2)$. For the nuclear norm, $\calM_0$ is the manifold of rank-$r$ matrices. Thus arguing as in \cite[Theorem~5]{koltchinskii11} (who use the Varshamov-Gilbert lemma~\cite{massartbook} to find the covering set), one can show that the minimax risk lower bound over $\calM_0$ is $C\sigma^2 r/n$. In view of Corollary~\ref{cor:SOIlinnuc}, we deduce that $\thewa$ and $\thvar$ are nearly minimax over the constant rank manifolds. 
\end{example}

\appendix
\section{Pre-requisites from convex analysis} \label{sec:convana}

We here collect some ingredients from convex analysis that are essential to our exposition.

\paragraph{Monotone conjugate}
\begin{lemma}
\label{lem:monconj}
Let $g$ be a non-decreasing function on $\RR_+$ that vanishes at $0$. Then the following hold:
\begin{enumerate}[label=(\roman*)]
\item $g^+$ is a proper closed convex and non-decreasing function on $\RR_+$ that vanishes at $0$. 
\item If $g$ is also closed and convex, then $g^{++}=g$.
\item Let $f: t \in \RR \mapsto g(|t|)$ such that $f$ is differentiable on $\RR$, where $g$ is finite-valued, strictly convex and strongly coercive. Then $g^+$ is likewise finite-valued, strictly convex, strongly coercive, and $f^*=g^+ \circ |\cdot|$ is differentiable on $\RR$. In particular, both $g$ and $g^+$ are strictly increasing on $\RR_+$.
\end{enumerate}
\end{lemma}

\begin{proof}
\begin{enumerate}[label=(\roman*)]
\item By~\cite[Proposition~13.11]{bauschke2011convex}, $g^+$ is a closed convex function. We have $\inf_{t \geq 0} g(t)=-\sup_{t \geq 0} t \cdot 0 - g(t) = -g^+(0)$. Since $g$ is non-decreasing and $g(0)=0$, then $g^+(0)=-\inf_{t \geq 0} g(t)=-g(0)=0$. In addition, by \eqref{eq:fenineq}, we have $g^+(a) \geq a \cdot 0 - g(0)=0$, $\forall a \in \RR_+$. This shows that $g^+$ is non-negative and $\dom(g^+) \neq \emptyset$, and in turn, it is also proper.

Let $a,b$ in $\RR_+$ such that $a < b$. Then
\[
g^+(a)-g^+(b) = (\sup_{t \geq 0} t a - g(t)) - (\sup_{t' \geq 0} t' b - g(t')) \leq \sup_{t \geq 0} (t a - g(t) - t b + g(t)) = \sup_{t \geq 0} t(a-b) = 0.
\]
That is, $g^+$ is non-decreasing on $\RR_+$.

\item This follows from~\cite[Theorem~12.4]{rockafellar1996convex}.

\item By definition of $f$, $f$ is a finite-valued function on $\RR$, strictly convex, differentiable and strognly coercive. It then follows from~\cite[Corollary~X.4.1.4]{hiriart1996convex} that $f^*$ enjoys the same properties. In turn, using the fact that both $f$ and $f^*$ are even, we have $g^+$ is strongly coercive, and strict convexity of $f$ (resp. $f^*$) is equivalent to that of $g$ (resp. $g^+$). Altogether, this shows the first claim. We now prove that $g$ vanishes only at $0$ (and similary for $g^+$). As $g$ is non-decreasing and strictly convex, we have, for any $\rho \in ]0,1[$ and $a,b$ in $\RR_+$ such that $a < b$,
\[
g(a) \leq g(\rho a + (1-\rho) b) < \rho g(a) + (1-\rho) g(b) \leq \rho g(b) + (1-\rho) g(b) = g(b) .
\]
\end{enumerate}
\end{proof}

\paragraph{Support function}
The \emph{support function} of $\calC \subset \RR^p$ is
\[
	\sigma_{\calC}(\bsom)=\sup_{\xx \in \calC} \dotp{\bsom}{\xx} .
\]
We recall the following properties whose proofs can be found in e.g.~\cite{rockafellar1996convex,hiriart1996convex}.
\begin{lemma}
\label{lem:suppcompact}
Let $\calC$ be a non-empty set. 
\begin{enumerate}[label=(\roman*)] 
\item $\sigma_\calC$ is proper lsc and sublinear.
\item $\sigma_\calC$ is finite-valued if and only if $\calC$ is bounded.
\item If $0 \in \calC$, then $\sigma_\calC$ is non-negative. 
\item If $\calC$ is convex and compact with $0 \in \interop(\calC)$, then $\sigma_\calC$ is finite-valued and coercive.
\end{enumerate}
\end{lemma}

\paragraph{Gauges and polars}

\begin{definition}[Polar set]\label{defn:convex-polarset}
  Let $\calC$ be a nonempty convex set. The set $\calC^\circ$ given by
  \begin{equation*}
    \calC^\circ = \enscond{\bseta \in \RR^p}{\dotp{\bseta}{\xx} \leq 1 \quad \forall \xx \in \calC}
  \end{equation*}
  is called the \emph{polar} of $\calC$.
\end{definition}
The set $\calC^\circ$ is closed convex and contains the origin. When $\calC$ is also closed and contains the origin, then it coincides with its bipolar, i.e. $\calC^{\circ\circ}=\calC$.

Let $\calC \subseteq \RR^p$ be a non-empty closed convex set containing the origin. The \emph{gauge} of $\calC$ is the function $\gauge_\calC$ defined on $\RR^p$ by
\begin{equation*}
  \gauge_\calC(\xx) =
  \inf \enscond{\lambda > 0}{\xx \in \lambda \calC} .
\end{equation*}
As usual, $\gauge_\calC(\xx) = + \infty$ if the infimum is not attained. 

Lemma~\ref{lem:convex-gauge} hereafter recaps the main properties of a gauge that we need. In particular, (ii) is a fundamental result of convex analysis that states that there is a one-to-one correspondence between gauge functions and closed convex sets containing the origin. This allows to identify sets from their gauges, and vice versa.
\begin{lemma}\label{lem:convex-gauge}
{~\\}\vspace*{-0.5cm}
\begin{enumerate}[label=(\roman*)]
\item $\gauge_\calC$ is a non-negative, lsc and sublinear function.
\item $\calC$ is the unique closed convex set containing the origin such that
\[
\calC = \enscond{\xx \in \RR^p}{\gamma_\calC(\xx) \leq 1} .
\]
\item $\gamma_\calC$ is finite-valued if, and only if, $0 \in \interop(\calC)$, in which case $\gamma_\calC$ is Lipschitz continuous.
\item $\gamma_\calC$ is finite-valued and coercive if, and only if, $\calC$ is compact and $0 \in \interop(\calC)$.
\end{enumerate}
\end{lemma}
See \cite{vaiterimaiai13} for the proof. 

Observe that thanks to sublinearity, local Lipschitz continuity valid for any finite-valued convex function is streghthned to global Lipschitz continuity. Moreover, $\gamma_\calC$ is a norm, having $\calC$ as its unit ball, if and only if $\calC$ is bounded with nonempty interior and symmetric.
 
We now define the polar gauge.
\begin{definition}[Polar Gauge]\label{defn:polar-gauge}
The polar of a gauge $\gamma_\calC$ is the function $\gamma_\calC^\circ$ defined by
\[
	\gamma_\calC^\circ(\bsom) = \inf\enscond{\mu \geq 0}{\dotp{\xx}{\bsom} \leq \mu \gamma_\calC(\xx), \forall \xx} .
\] 
\end{definition}
An immediate consequence is that gauges polar to each other have the property
\begin{equation}\label{eq:dualineq}
    \dotp{\xx}{\bsu} \leq \gauge_\calC(\xx) \gauge_\calC^\circ(\bsu) \quad \forall (\xx,\bsu) \in \dom(\gauge_\calC) \times \dom(\gauge_\calC^\circ) ,
\end{equation}
just as dual norms satisfy a duality inequality. In fact, polar pairs of gauges correspond to the best inequalities of this type.

\begin{lemma}\label{lem:convex-polar-gauge}
  Let $\calC \subseteq \RR^p$ be a closed convex set containing the origin.
  Then,
  \begin{enumerate}[label=(\roman*),start=2]
  \item $\gauge_\calC^\circ$ is a gauge function and $\gamma_\calC^{\circ\circ}=\gamma_\calC$.
  \item $\gamma_\calC^\circ=\gamma_{\calC^\circ}$, or equivalently
  \[
  \calC^\circ = \enscond{\xx \in \RR^p}{\gamma_\calC^\circ(\xx) \leq 1} .
  \]
\item The gauge of $\calC$ and the support function of $\calC$ are mutually polar, i.e.
  \begin{equation*}
    \gauge_\calC = \sigma_{\calC^\circ} \qandq \gauge_{\calC^\circ} = \sigma_\calC ~.
  \end{equation*}
\end{enumerate}
\end{lemma}
See \cite{rockafellar1996convex,hiriart1996convex,vaiterimaiai13} for the proof.\\

\section{Expectation of the inner product} \label{sec:avdotp}

We start with some definitions and notations that will be used in the proof. For a non-empty closed convex set $\calC \in \RR^p$, we denote $\mnsel{\calC}$ its minimal selection, i.e. the element of minimal norm in $\calC$. This element is of course unique. For a proper lsc and convex function $f$ and $\gamma > 0$, its Moreau envelope (or Moreau-Yosida regularization) is defined by
\begin{align*}
\env{f}{\gamma}(\xx) &\eqdef \min_{\bxx \in \RR^p} \frac{1}{2\gamma}\norm{\bxx - \xx}{2}^2 + f(\bxx) .
\end{align*}

The Moreau envelope enjoys several important properties that we collect in the following lemma.
\begin{lemma}\label{lem:envfb}
Let $f$ be a finite-valued and convex function. Then 
\begin{enumerate}[label=(\roman*),ref=(\roman*)]
\item $\pa{\env{f}{\gamma}(\bstheta)}_{\gamma > 0}$ is a decreasing net, and $\forall \xx \in \RR^p$, $\env{f}{\gamma}(\xx) \nearrow f(\xx)$ as $\gamma \searrow 0$.
\label{lem:envfb1}
\item $\env{f}{\gamma} \in C^1(\RR^p)$ with $\gamma^{-1}$-Lipschitz continuous gradient.
\label{lem:envfb2}
\item $\forall \xx \in \RR^p$, $\nabla \env{f}{\gamma}(\xx) \to \mnsel{\partial f(\xx)}$ and $\norm{\nabla\env{f}{\gamma}(\xx)}{2} \nearrow \norm{\mnsel{\partial f(\xx)}}{2}$ as $\gamma \searrow 0$.
\label{lem:envfb3}
\end{enumerate}
\end{lemma}

\begin{proof}
\ref{lem:envfb1} \cite[Proposition~12.32]{bauschke2011convex}. \ref{lem:envfb2} \cite[Proposition~12.29]{bauschke2011convex}. \ref{lem:envfb3} By assumption, $f$ is subdifferentiable everywhere and its subdifferential is a maximal monotone operator with domain $\RR^p$, and the result follows from \cite[Corollary~23.46(i)]{bauschke2011convex}. 
\end{proof}

We are now equipped to prove the following important result\footnote{The result will be proved using Moreau-Yosida regularization. Yet another alternative proof could be based on mollifiers for approximating subdifferentials.}.
\begin{proposition}\label{prop:avdotp}
Let the density $\mu_n$ in \eqref{eq:EWAaggregator}, where 
\begin{enumerate}[label=(\alph*)]
\item $F$ satisfies Assumptions~\ref{assump:F}-\ref{assump:Fmom};
\item $J$ is a finite-valued lower-bounded convex function, and $\exists R > 0$ and $\rho \geq 0$, such that $\forall \xx \in \RR^p$, $\norm{\mnsel{\partial J(\xx)}}{2} \leq R \anorm{\xx}{2}^\rho$; 
\item and $V_n$ is coercive. 
\end{enumerate}
Then, $\forall \bxx \in \RR^p$,
\[
\EE{\mu_n}{\dotp{\mnsel{\partial V_n(\xx)}}{\bxx-\xx}} = -p\beta .
\]
\end{proposition}

This result covers of course the situation where $J$ fulfills \ref{assump:J}. In this case, since $\partial J(\xx) \subset \calC^\circ$ by Theorem~\ref{thm:decomp}(i), we have $\rho=0$ and $R=\diam(\calC^\circ)$, the diameter of the convex compact set $\calC^\circ$ containing the origin. It can be shown that, when $F(\cdot,\yy)$ is strongly coercive, the coercivity assumption (c) can be equivalently stated as $J_{\infty}(\xx) > 0$, $\forall \xx \in \ker(\XX) \setminus \ens{0}$, where $J_\infty$ is the recession/asymptotic function of $J$; see e.g.~\cite{rockafellar1998variational}. 

\begin{proof}
Let $V^\gamma_n(\xx) \eqdef \tfrac{1}{n}F(\XX\xx,\yy)+\lambda_n \env{J}{\gamma}(\xx)$ and define $\mu^{\gamma}_n(\xx) \eqdef \exp\pa{-V^\gamma_n(\xx)/\beta}/Z$, where $0 < Z < +\infty$ is the normalizing constant of the density $\mu_n$. Assumption~\ref{assump:F} and Lemma~\ref{lem:envfb}\ref{lem:envfb2}-\ref{lem:envfb3} tell us that $V^\gamma_n \in C^1(\RR^p)$ and $\nabla V^\gamma_n(\xx) \to \mnsel{\partial V_n(\xx)}$ as $\gamma \to 0$. Thus
\begin{align*}
\EE{\mu_n}{\dotp{\mnsel{\partial V_n(\xx)}}{\bxx-\xx}} 
		&= \int_{\RR^p} \lim_{\gamma \to 0} \dotp{\mu^\gamma_n(\xx)\nabla V^\gamma_n(\xx)}{\bxx-\xx} d\xx .
\end{align*}
We now check that $\dotp{\mu^\gamma_n(\xx)\nabla V^\gamma_n(\xx)}{\bxx-\xx}$ is dominated by an integrable function. From the definition of the Moreau envelope, we have
\[
V^\gamma_n(\xx) = \min_{\bxx \in \RR^p} \tfrac{1}{n}F(\XX\xx,\yy) + \lambda_n\bpa{J(\xx - \bxx) + \frac{1}{2\gamma}\norm{\bxx}{2}^2} .
\]
From coercivity of $V_n$, the objective in the $\min$ is also coercive in $(\xx,\bxx)$ by \cite[Exercise~3.29(b)]{rockafellar1998variational}. It then follows from \cite[Theorem~3.31]{rockafellar1998variational} that $V^\gamma_n$ is also coercive. In turn, \cite[Theorem~11.8(c) and 3.26(a)]{rockafellar1998variational} allow to assert that for some $a \in  ]0,+\infty[$, $\exists b \in ]-\infty,+\infty[$ such that for all $\gamma > 0$ and $\xx \in \RR^p$
\begin{equation}\label{eq:bndmugam}
\mu^{\gamma}_n(\xx) \leq \exp\pa{-a\anorm{\xx}{2}-b}/Z .
\end{equation}
Lemma~\ref{lem:envfb}-\ref{lem:envfb3} and assumption (b) on $J$ entail that for any $\xx \in \RR^p$,
\[
\norm{\nabla\env{J}{\gamma}(\xx)}{2} \leq \norm{\mnsel{\partial J(\xx)}}{2} \leq R \anorm{\xx}{2}^\rho .
\]
Altogether, we have
\begin{align*}
\abs{\dotp{\mu^\gamma_n(\xx)\nabla V^\gamma_n(\xx)}{\bxx-\xx}} 
		&\leq \mu^\gamma_n(\xx) \pa{\abs{\dotp{\XX^\top \tfrac{1}{n}\nabla F(\XX\xx,\yy)}{\bxx-\xx}}+\lambda_n\norm{\nabla\env{J}{\gamma}(\xx)}{2}\norm{\bxx-\xx}{2}} \\
		&\leq C Z^{-1}\exp\pa{-F(\XX\xx,\yy)/(n\beta)} \abs{\dotp{\tfrac{1}{n}\nabla F(\XX\xx,\yy)}{\XX(\bxx-\xx)}} \\
		& \qquad \qquad + (Z\exp{b})^{-1} \lambda_n R \exp\pa{-a\anorm{\xx}{2}} \norm{\xx}{2}^\rho\norm{\bxx-\xx}{2} ,
\end{align*}
where the constant $C > 0$ reflects the lower-boudedness of $J$.
It is easy to see that the function in this upper-bound is integrable, where we also use \ref{assump:Fmom}. Hence, we can apply the dominated convergence theorem to get
\begin{align*}
\EE{\mu_n}{\dotp{\mnsel{\partial V_n(\xx)}}{\bxx-\xx}} &= \lim_{\gamma \to 0} \int_{\RR^p} \dotp{\mu^\gamma_n(\xx)\nabla V^\gamma_n(\xx)}{\bxx-\xx} d\xx .
\end{align*}
Now, by simple differential calculus (chain and product rules), we have
\begin{align*}
\dotp{\mu^\gamma_n(\xx)\nabla V^\gamma_n(\xx)}{\bxx-\xx} 
		&= -\beta\dotp{\nabla \mu^\gamma_n(\xx)}{\bxx-\xx} \\
		&= -\beta \sum_{i=1}^p \frac{\partial}{\partial \xx_i}\pa{\mu^\gamma_n(\xx)(\bxx_i-\xx_i)} - p\beta \mu^\gamma_n(\xx) .	
\end{align*}
Integrating the first term, we get by Fubini theorem and the Newton-Leibniz formula
\begin{align*}
\int_{\RR^{p-1}} \pa{\int_{\RR} \frac{\partial}{\partial \xx_i}\pa{\mu^\gamma_n(\xx)(\bxx_i-\xx_i)} d\xx_i} &d\xx_1 \cdots d\xx_{i-1} d\xx_{i+1}\cdots d\xx_p \\ 
	&= \int_{\RR^{p-1}} \lrsquare{\mu^\gamma_n(\xx)(\bxx_i-\xx_i)}_{\RR} d\xx_1 \cdots d\xx_{i-1} d\xx_{i+1}\cdots d\xx_p = 0 ,
\end{align*}
where we used coercivity of $V^\gamma_n$ (see \eqref{eq:bndmugam}) to conclude that $\lim_{|\xx_i| \to +\infty} \mu^\gamma_n(\xx)(\bxx_i-\xx_i) = 0$. For the second term, we have from Lemma~\ref{lem:envfb}\ref{lem:envfb1} that $\mu^\gamma_n \to \mu_n$ as $\gamma \to 0$. Thus, arguing again as in \eqref{eq:bndmugam}, we can apply the dominated convergence theorem to conclude that 
\[
\lim_{\gamma \to 0} \int_{\RR^p} \mu^\gamma_n(\xx) d\xx = \int_{\RR^p} \mu_n(\xx) d\xx = 1.
\]
This concludes the proof.
\end{proof}


\paragraph{Acknowledgement.}
This work was supported by Conseil R\'egional de Basse-Normandie and partly by Institut Universitaire de France.

\bibliographystyle{abbrv}
\bibliography{bibli}

\end{document}